\documentclass[a4paper]{article}

\usepackage{graphicx,amssymb}
\usepackage{amsmath}
\usepackage{cases}
\usepackage{float}
\usepackage{listings}
\usepackage{times}
\usepackage{tikz}
\usepackage{epic}
\usepackage{titlesec}
\usepackage{geometry}
\usepackage{amsthm}
\usepackage{bm}
\usepackage{stmaryrd}
\usepackage{subfigure}
\usepackage{indentfirst}
\usepackage{xcolor}
\usepackage{epstopdf}
\usepackage{enumerate}
\usepackage{enumitem}
\usepackage{hyperref}

\definecolor{red}{rgb}{1.00,0.00,0.00}
{\numberwithin{equation}{section}
\setlength{\parindent}{1em}

\newtheorem{theorem}{Theorem}[section]
\newtheorem{lemma}{Lemma}[section]

\newtheorem{example}{Example}[section]

\SetSymbolFont{stmry}{bold}{U}{stmry}{m}{n}

\newcommand{\normmm}[1]{{\left\vert\kern-0.25ex\left\vert
\kern-0.25ex\left\vert #1
    \right\vert\kern-0.25ex\right\vert\kern-0.25ex\right\vert}}
\geometry{left=3cm,right=3cm,top=4cm,bottom=2.5cm}

\begin{document}
\title{An analysis of the NLMC upscaling method for high contrast problems}
\author{ Lina Zhao\footnotemark[1]\qquad
\;Eric T. Chung\footnotemark[2]}
\renewcommand{\thefootnote}{\fnsymbol{footnote}}
\footnotetext[1]{Department of Mathematics,The Chinese University of Hong Kong, Hong Kong Special Administrative Region. ({lzhao@math.cuhk.edu.hk})}
\footnotetext[2]{Department of Mathematics,The Chinese University of Hong Kong, Hong Kong Special Administrative Region. ({tschung@math.cuhk.edu.hk})}

%
\maketitle

\textbf{Abstract:}
In this paper we propose simple multiscale basis functions with constraint energy minimization to solve elliptic problems with high contrast medium. Our methodology is based on the recently developed non-local
multicontinuum method (NLMC). The main ingredient of the method is the construction of suitable local basis functions with the capability of capturing multiscale features and non-local effects. In our method, each coarse block is decomposed into various regions according to the contrast ratio, and we require that the contrast ratio should be relatively small within each region. The basis functions are constructed by solving a local problem defined on the oversampling domains and they have mean value one on the chosen region and zero mean otherwise. Numerical analysis shows that the resulting basis functions can be localizable and have a decay property. The convergence of the multiscale solution is also proved. Finally, some numerical experiments are carried out to illustrate the performances of the proposed method. They show that the proposed method can solve problem with high contrast medium efficiently. In particular, if the oversampling size is large enough, then we can achieve the desired error.


\pagestyle{myheadings} \thispagestyle{plain} \markboth{Zhao and Chung}
    {Constraint energy minimization for high contrast problem}

\section{Introduction}

In this paper we consider
\begin{equation}
\begin{split}
-\nabla \cdot(\kappa \nabla u)&=f \quad\mbox{in}\; \Omega,\\
u&=0 \quad\mbox{on}\; \partial \Omega,
\end{split}
\label{eq:model}
\end{equation}
where $\Omega\subset \mathbb{R}^2$ is the computational domain and $\kappa$ is a high contrast with $0< \kappa_{min}\leq \kappa\leq \kappa_{max}$
and is a multiscale field. The proposed method can be extended to 3D easily.

If the coefficient $\kappa$ is rough, then the solution $u$ to \eqref{eq:model} will also be rough; to be specific, $u$ will not in general be in $H^2(\Omega)$ and may not be in $H^{1+\epsilon}(\Omega)$ for any $\epsilon>0$. For this kind of low regularity, standard analysis usually fails. Moreover, the classical polynomial based finite element methods could perform arbitrary badly for such problems, see, e.g., \cite{Babuska00}. To resolve this issue, various numerical methods have been proposed and analyzed, and among all the methods we mention in particular the special finite element methods \cite{Babuska83,Babuska94}, the upscaled models \cite{Durlofsky91,WuEffndievHou02} and the multiscale methods \cite{Hughes95,HouWu97,Hughes98,HouWuCai99,Arbogast07,ChungEfendievLi14,ChungEfendievleungflow,CEH16,Owhadi17,OwhadiZhang07,Gallistl16,Gallistl17}.

The concept of non-local upscaling has been successfully applied to problems in porous media, see, e.g., \cite{EfendievDurlofsky00,Cushman02,Durlofsky07}. Motivated by the work given in \cite{Delgoshaie15}, the nonlocal multicontinua (NLMC) upscaling technique was initially introduced for flows in heterogeneous fractured media in \cite{ChungEfendiev18}, and have been successfully applied to different problems under application \cite{VasilyevaChung19,VasilyevaChung192,VasilyevaChungmixed18,VasilyevaChungnlmc19}. The main idea of NLMC upscaling technique is to construct the multiscale basis functions over the oversampling domain via an energy minimization principle. Note that the constraint should be chosen properly in order to make the localization possible. One distinctive feature of the method is that it allows a systematic upscaling for processes in the fractured porous media, and provides an effective coarse scale model whose degrees of freedom have physical meaning.

Inspired by the work given in \cite{ChungEfendiev18,VasilyevaChung19,VasilyevaChung192}, the goal of this paper is to extend the idea of nonlocal multicontinua to problem \eqref{eq:model}. For our approach, we start with decomposing the coarse block into different regions and the criterion used for the decomposition is to have relatively small contrast ratio within each region. Then, we define the constraint energy minimzation problem in the oversampling domain, where the restriction for the basis functions is defined such that they have mean value one in the chosen region and zero mean otherwise, in addition the basis functions vanish on the boundary of the oversampling domain. We remark that the vanishing property is important for the localization of the multiscale basis functions and the localization idea has also been exploited in \cite{Malqvbist14} to solve problems with heterogeneous and highly varying coefficients. Next, we can solve the local minimization problem by using the equivalent saddle point formulation to achieve the multiscale basis functions. The resulting multiscale basis functions have decay property, in addition, it can capture the fine-grid information well provided proper number of overampling layers are chosen. With the multiscale basis functions, we can solve the upscaled equation to obtain the upscaled coarse grid solution. It is worth mentioning that in our method the number of basis function is relatively small and it is equal to the number of scales over the domain. We also analyze the convergence of the proposed method. For this, we first compare the difference between the multiscale basis functions and the global basis functions, combining this with the convergence of the global solution, then we can prove the convergence of the multiscale solution in $L^2$ norm and weighted energy norm. The analysis indicates that the convergence rate only depends on the local contrast ratio, namely, the contrast ratio within each region. With proper number of oversampling layers, the first order convergence measured in energy norm can be obtained. Some numerical experiments are also carried out. The numerical experiments show that with the fixed coarse mesh size, the oversampling layers should be selected properly to achieve the desired error, in addition, for a fixed oversampling size, the performance of the scheme will deteriorate as the medium contrast increases.

The rest of the paper is organized as follows. In the next section, we present the construction of the proposed method for \eqref{eq:model}. The convergence analysis for the multiscale solution is proposed in Section~\ref{sec:error}. Then, some numerical experiments are investigated in Section~\ref{sec:numerical} to confirm the theoretical results. Finally, the conclusions are given in Section~\ref{sec:conclusion}.

\section{Preliminaries}

\subsection{Description of NLMC method}

The solution of \eqref{eq:model} satisfies
\begin{align}
a(u,v)=(f,v)\quad \forall v\in H^1_0(\Omega),\label{eq:weak}
\end{align}
where $a(u,v)=\int_\Omega \kappa \nabla u\cdot \nabla v\;dx$.

Next, the notations of the fine grids and coarse grids are introduced. Let $\mathcal{T}_H$ be a coarse-grid of the domain $\Omega$ and $\mathcal{T}_h$ be a conforming fine triangulation of $\Omega$. We assume that $\mathcal{T}_h$ is a refinement of $\mathcal{T}_H$, where $h$ and $H$ represent the fine and coarse mesh sizes, respectively. Let $K_i\in \mathcal{T}_H$ be the $i$-th coarse block and let $K_{i,m}$ be the corresponding oversampled region obtained by enlarging the coarse block $K_i$ by $m$ coarse grid layers (See Figure~\ref{grid} for an illustration). We let $N$ be the number of elements in $\mathcal{T}_H$. Furthermore, each coarse block $K_{i}, i=1\cdots,N$ is decomposed into different regions $K_i^{j}, j=1,\cdots,l_i$ and $l_i$ is the number of regions within coarse block $K_i$. In addition, we require that within each region $K_i^j$, $\kappa$ should satisfy $\{\kappa_0\leq\kappa\leq \kappa_1\}$ and the contrast ratio $C_{ratio}^{i,j}=\frac{\kappa_1}{\kappa_0}$ should be relatively small. In addition, we define $C_{ratio}=\max_{i,j}C_{ratio}^{i,j}$ for any $i=1,\cdots,N, j=1,\cdots,l_i$. We remark that each region $K^j_i$ is a continuum. 
\begin{figure}[H]
\centering
\includegraphics[width=6cm]{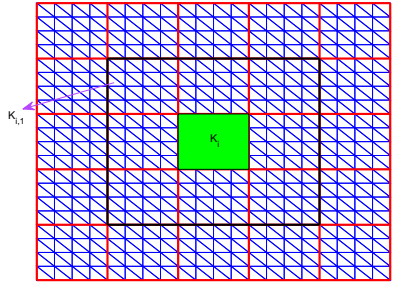}
\caption{Schematic of the coarse grid $K_i$, the oversampling region $K_{i,1}$ and the fine grids.}
\label{grid}
\end{figure}

Consider an oversampling region $K_{i,m}$ of the coarse block $K_i$, then the multiscale basis function $\psi_{i,ms}^{(j)}\in H^1_0(K_{i,m})$ is constructed by minimizing $a(\psi_{i,ms}^{(j)},\psi_{i,ms}^{(j)})$ subject to the following conditions
\begin{align*}
\frac{1}{|K_l^{n}|}\int_{K_l^{n}}\psi_i^{(j)}=\delta_{li}\delta_{nj}\quad \forall K_l^n\subset K_{i,m},
\end{align*}
where $\delta_{li},\delta_{nj}$ is the Dirac delta function and $|K_l^{n}|$ denotes the area of $K_l^n$. We can see that $\psi_i^{(j)}$ has mean value $1$ on the $j$-th region within the coarse block $K_i$ and $0$ mean in other regions inside the oversampling domain.

We remark that the above minimization problem is implicit, to solve it explicitly, we can write down the following equivalent variational formulation over each $K_{i,m}$:
\begin{align}
a(\psi_i^j,v)+\sum_{K_l^n\subset K_{i,m}}\lambda_l^n\int_{K_l^n}v\;dx&=0\quad\forall v\in H^1_0(K_{i,m}),\label{multiscale1}\\
\int_{K_l^n}\psi_i^j\;dx&=\int_{K_l^n}\delta_{li}\delta_{nj}\;dx\quad \forall K_l^n\subset K_{i,m},\label{multiscale2}
\end{align}
where $\lambda_l^n\in Q_h(K_{i,m})$ and $Q_h$ is a piecewise constant function
with respect to each region $K_i^j, i=1,\cdots, N, j=1,\cdots,l_i$ of
$\Omega$, and $Q_h(K_{i,m})$ denotes $Q_h$ restricted to $K_{i,m}$.
An illustration of the multiscale basis functions can be found in Figure~\ref{multiscale-basis}.


Then we obtain our multiscale space
\begin{align*}
V_{ms}=\mbox{span}\{\psi_{i,ms}^{(j)}\}.
\end{align*}
The resulting coarse grid equation can be written as
\begin{align*}
a(\bar{u},v)=(f,v)\quad \forall v\in V_{ms}.
\end{align*}

The construction of the local multiscale basis function is motivated by the global basis construction as defined below, and in the subsequent analysis we will exploit the global basis functions to show the convergence analysis. The global basis function $\psi_i^{(j)}$ is defined by
\begin{align}
\psi_i^{(j)}=\arg\min \{a(q_i^{(j)},q_i^{(j)})| q_i^{(j)}\in H^1_0(\Omega), \quad \frac{1}{|K_l^n|}\int_{K_l^n}q_i^{(j)}\;dx=\delta_{li}\delta_{nj},\;\forall K_l^n\subset \Omega \}.\label{global-multiscale}
\end{align}
Out multiscale finite element space $V_{glo}$ is defined by
\begin{align*}
V_{glo}=\mbox{span}\{\psi_i^{(j)}|1\leq i\leq N,\; 1\leq j\leq l_i\}.
\end{align*}

For later analysis, we define $\pi_{ij}(v)$ to be the projection which is defined for each region $K_i^j$ as
\begin{align*}
\pi_{ij}(v) =\frac{1}{|K_i^j|}\int_{K_i^j} v\;dx\quad \forall v\in L^2(\Omega)
\end{align*}
and
\begin{align*}
\pi(v)= \sum_{i=1}^N\sum_{j=1}^{l_i}\pi_{ij}(v).
\end{align*}
In addition, we define $\tilde{V}$ as the null space of the projection $\pi$, namely, $\tilde{V}=\{v\in H^1_0(\Omega)|\pi(v)=0\}$. Then for any $\psi_i^{(j)}\in V_{glo}$, we have
\begin{align*}
a(\psi_i^{(j)},v)=0\quad \forall v\in \tilde{V}.
\end{align*}
We remark that $\tilde{V}=V_{glo}^\perp$ and interested readers can refer to \cite{ChungEfendievleung18} for the explanations.

The approximate solution $u_{glo}\in V_{glo}$ obtained in the global multiscale space $V_{glo}$ is defined by
\begin{align}
a(u_{glo},v)=(f,v)\quad \forall v\in V_{glo}\label{eq:global}.
\end{align}

For later analysis, we define $\|v\|_{a}^2=\int_\Omega \kappa |\nabla u|^2\;dx$. In addition, for a given subdomain $\Omega_i\subset \Omega$, we define the local $a$-norm by $\|v\|_{a(\Omega_i)}^2=\int_{\Omega_i}\kappa |\nabla v|^2\;dx$.

\begin{figure}
\centering
\scalebox{0.3}{
\includegraphics[width=20cm]{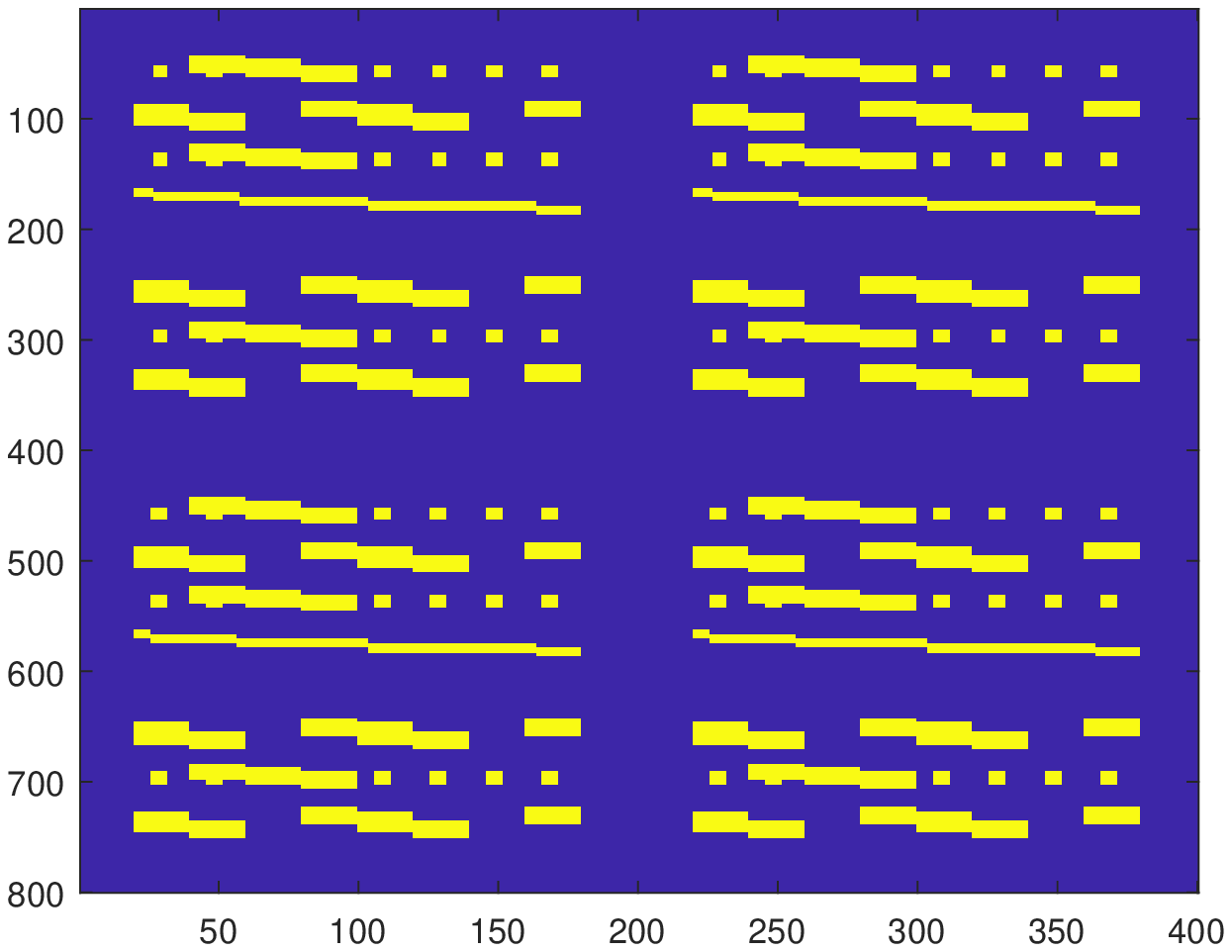}
}
\scalebox{0.3}{
\includegraphics[width=20cm]{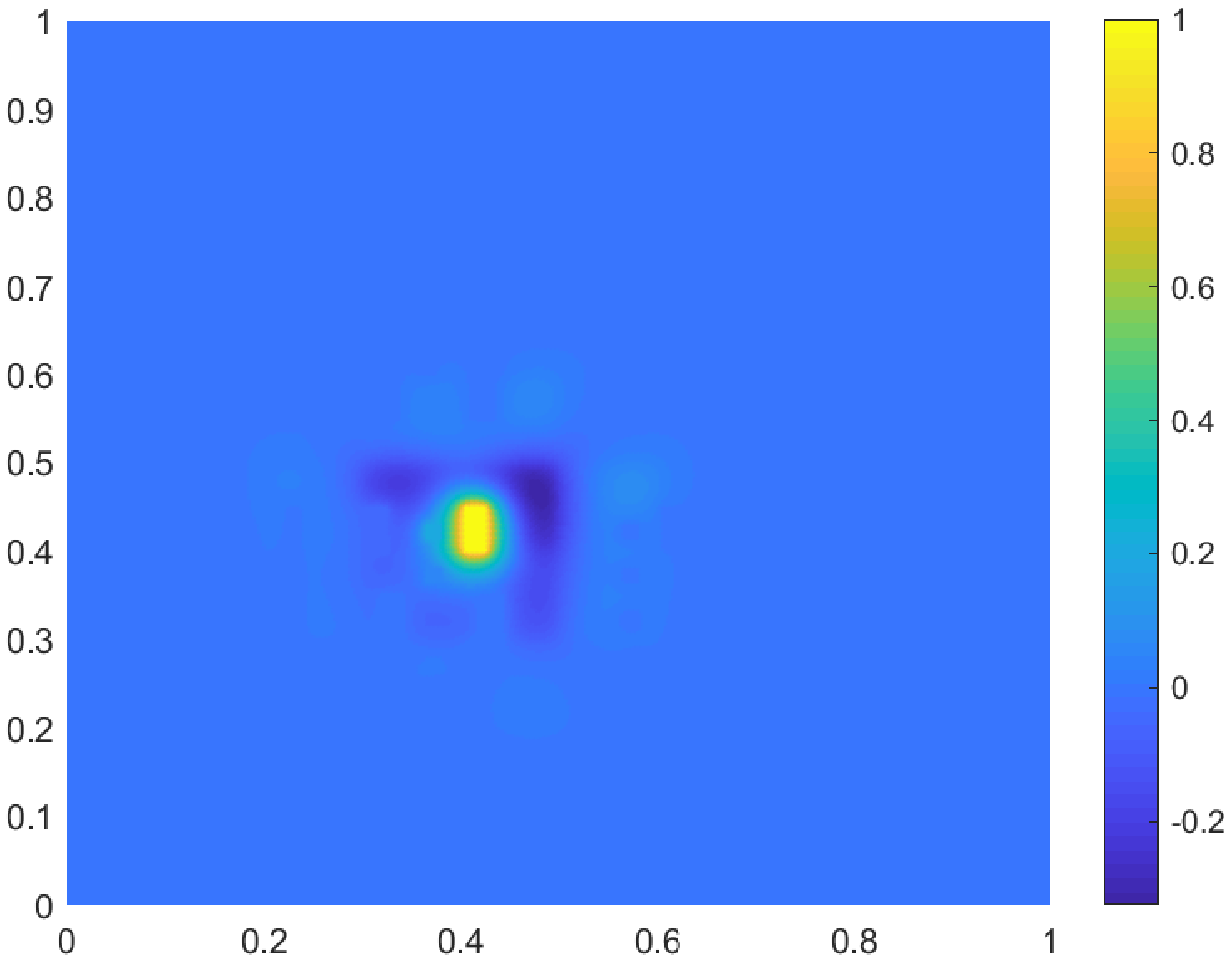}
}
\caption{An illustration of the decay property of the
 multiscale basis function. Left: a high contrast medium.
  Right: a multiscale basis function.}
\label{multiscale-basis}
\end{figure}

\subsection{Computational issue}

For the convenience of the readers, we write down the implementation of the proposed method as follows.
\begin{enumerate}
  \item Calculate the multiscale basis functions $\psi_{i,ms}^{(j)}$ by solving \eqref{multiscale1}-\eqref{multiscale2} for each region $K_i^j, i=1\cdots,N, j=1,\cdots,l_i$.
  \item Generate the projection matrix
  \begin{align*}
  R^T=[\psi_{1}^{(1)}\cdots,\psi_{1,ms}^{(l_1)},\cdots,\psi_{N,ms}^{(1)},\cdots,\psi_{N,ms}^{(l_N)}],
  \end{align*}
  where $\psi_{i,ms}^{(j)}$ is a column vector using its representation in the fine grid.
  \item Construct the coarse grid system
  \begin{align*}
  RAR^T \bar{u}=Rb
  \end{align*}
  and solve the above equation to get $\bar{u}$.
\end{enumerate}
Note that the downscale solution can be defined by $u_{ms}=R^T \bar{u}$. Our coarse grid solutions have physical meaning, which is the average value of the solution on each region $K_i^j$.

\section{Error analysis}\label{sec:error}

In this section, we will carry out the error analysis for the proposed method. We first show the convergence of the global basis function defined in \eqref{global-multiscale}, then we show the decay property of the local multiscale basis function, using which the convergence of the multiscale solution can be obtained.

\subsection{Convergence}

This subsection presents the convergence of the approximate solution obtained in \eqref{eq:weak} as stated in the next lemma.
\begin{lemma}\label{lemma:conglo}
Let $u$ be the solution in \eqref{eq:weak} and $u_{glo}$ be the solution in \eqref{eq:global},
then we have
\begin{align*}
\|u-u_{glo}\|_a&\leq C HC_{ratio}^{1/2}\|\kappa^{-1/2}f\|_0.
\end{align*}

\end{lemma}

\begin{proof}

By the definitions of $u$ and $u_{glo}$, we have
\begin{align*}
a(u,v)&=(f,v)\quad \forall v\in H^1_0(\Omega),\\
a(u_{glo},v)&=(f,v)\quad \forall v\in V_{glo}.
\end{align*}
Combining these two equations, we can get
\begin{align*}
a(u-u_{glo},v)=0\quad \forall v\in V_{glo}.
\end{align*}
So, we have $u-u_{glo}\in V_{glo}^{\perp}=\tilde{V}$. It then follows that
\begin{align*}
a(u-u_{glo},u-u_{glo})=a(u,u-u_{glo})=(f,u-u_{glo})\leq \|\kappa^{-1/2}f\|_0\|\kappa^{1/2}(u-u_{glo})\|_0,
\end{align*}
Since $\pi(u-u_{glo})=0$, the Poincar\'{e} inequality yields
\begin{align*}
\int_{K_i^j}(u-u_{glo})^2\leq C H^2\int_{K_i^j}|\nabla (u-u_{glo})|^2.
\end{align*}
Therefore, the preceding arguments reveal that
\begin{align*}
\|u-u_{glo}\|_a^2\leq CHC_{ratio}^{1/2}\|\kappa^{-1/2}f\|_0 \|u-u_{glo}\|_a,
\end{align*}
which gives the desired estimate.


\end{proof}

\subsection{Decay property of the multiscale basis functions}

This section aims to proving the global basis functions are localizable. To this end,
for each coarse block $K$, we define $B$ to be a bubble function and $B\mid_{\tau}=\frac{\varphi_1\varphi_2\varphi_3}{27},\forall \tau\in \mathcal{T}_h(K)$, where $\varphi_i$ is barycentric coordinates and $\mathcal{T}_h(K)$ denotes the fine grids restricted to $K$, and more information regarding the bubble function $B$ can be found in \cite{Verfurth96}.


The next lemma considers the following minimization problem defined on a coarse block $K_i$:
\begin{align}
v_{i}^{(j)}=\arg\min\{a(q_i^{(j)},q_i^{(j)})|q_i^{(j)}\in H^1_0(K_i), \pi_{il}(q_i^{(j)})=v_{aux}\; \forall l,=1\cdots,l_i\}\label{eq:minimizationK}
\end{align}
for a given $v_{aux}\in Q_h(K_i)$.

\begin{lemma}\label{lemma:saddle}
For all $v_{aux}\in Q_h$, there exists a function $v\in H^1_0(\Omega) $ such that
\begin{align*}
\pi(v)=v_{aux},\quad \|v\|_a^2\leq D\|\kappa ^{1/2}v_{aux}\|_0^2,\quad \mbox{supp}(v)\subset \mbox{supp}(v_{aux}).
\end{align*}

\end{lemma}

\begin{proof}
Let $v_{aux}\in Q_h(K_i)$. The minimization problem is equivalent to the following variational problem: find $v_i^{(j)}\in H^1_0(K_i)$ and $\mu\in Q_h(K_i)$ such that
\begin{align}
a_i(v_i^{(j)},w)+\sum_{K_i^l\subset K_i}\mu_l\int_{K_i^l}w\;dx
&=0\quad \forall w\in H^1_0(K_i),\label{eq:minimizationK1}\\
\int_{K_i^l}v_i^{(j)}\;dx&=\int_{K_i^l}v_{aux}\;dx\quad \forall l=1,\cdots,l_i.\label{eq:minimizationK2}
\end{align}
Let $s_i(v,v_{aux})=\sum_{K_i^l\subset K_i}\int_{K_i^l}v v_{aux}\;dx$.
Note that, by the mixed finite element theory (cf. \cite{BrezziFortin91}), the well-posedness of the minimization problem
is equivalent to the existence of a function $v\in H^1_0(K_i)$ such that
\begin{align*}
s_i(v,v_{aux})\geq C \|v_{aux}\|_{0,K_i}^2, \quad \|v\|_{a(K_i)}\leq C \|v_{aux}\|_{0,K_i}.
\end{align*}
Note that $v_{aux}$ is supported in $K_i$. We let $v=Bv_{aux}$. By the definition of $s_i$, we have
\begin{align*}
s_i(v,v_{aux})=\sum_{K_i^l\subset K_i}\int_{K_i^l}Bv_{aux}^2\geq C \|v_{aux}\|_{0,K_i}^2.
\end{align*}
In addition,
\begin{align*}
\|v\|_{a(K_i)}^2=\|Bv_{aux}\|_{a(K_i)}^2\leq C \|v\|_{a(K_i)}\|\kappa^{1/2}v_{aux}\|_{0,K_i},
\end{align*}
Thus
\begin{align*}
\|v\|_{a(K_i)}\leq C \|\kappa^{1/2}v_{aux}\|_{0,K_i}
\end{align*}
and the minimization problem \eqref{eq:minimizationK} has a unique solution $v\in H^1_0(K)$. Therefore, $v$ and $v_{aux}$ satisfy \eqref{eq:minimizationK1}-\eqref{eq:minimizationK2}. From \eqref{eq:minimizationK2}, we can obtain $\pi_{il}(v)=v_{aux}$. The  assertion follows.

\end{proof}

The rest of this section attempts to estimating the difference between the global and multiscale basis
functions. For this purpose, we first introduce some notations used for the subsequent analysis.
We define the cutoff function with respect to these oversampling domains.
For each $K_i$, we recall that $K_{i,m}$ is the oversampling coarse region by
enlarging $K_i$ by $m$ coarse grid layers. For $M>m$, we
define $\chi_{i}^{M,m}\in \mbox{span}\{\chi_i^{ms}\}$ such that $0\leq \chi_i^{M,m}\leq 1$
and
\begin{align}
\chi_i^{M,m}&=1 \quad \mbox{in}\;K_{i,m},\label{cut1}\\
\chi_i^{M,m}&=0\quad \mbox{in}\;\Omega\backslash K_{i,M}\label{cut2}.
\end{align}
Note that we have $K_{i,m}\subset K_{i,M}$ and $\{\chi_{i}^{ms}\}_{i=1}^N$ are the standard  multiscale finite element (MsFEM) basis functions (cf. \cite{HouWu97}).

The next lemma shows the difference between the global and multiscale basis functions,
which will play an important role in the proof of the convergence of the multiscale solution.

\begin{lemma}\label{lemma:decay}
We consider the oversampled domain $K_{i,k}$ with $k\geq 2$. That is, $K_{i,k}$ is an oversampled region by enlarging $K_i$ by $k$ grid layers.
Let $\delta_{lj}$ be the Dirac delta function. We let $\psi_{i,ms}^{(j)}$ be the multiscale basis functions obtained in \eqref{multiscale1}-\eqref{multiscale2} and let $\psi_i^{(j)}$ be the global multiscale basis functions obtained in \eqref{global-multiscale}. Then we have
\begin{align*}
\|\psi_i^{(j)}-\psi_{i,ms}^{(j)}\|_a^2\leq CE\|\kappa^{1/2}\delta_{lj}\|_{0,K_i}\quad \forall l,j=1\cdots,l_i
\end{align*}
and
\begin{align}
E = D^2(1+C_{ratio}H^2)(1+\frac{1}{2D^{1/2}HC_{ratio}^{1/2}})^{1-k}.\label{eq:E}
\end{align}

\end{lemma}

\begin{proof}

For the given $\delta_{lj}\in Q_h$, by Lemma~\ref{lemma:saddle}, there exists a $\tilde{\phi}_i^{(j)}\in H^1_0(\Omega)$ such that
\begin{align}
\pi_{il}(\tilde{\phi}_i^{(j)})=\delta_{lj}, \quad \|\tilde{\phi}_i^{(j)}\|_a^2\leq D\|\kappa^{1/2}\delta_{lj}\|_0^2\quad \mbox{and}\quad \mbox{supp}(\tilde{\phi}_i^{(j)})\subset K_i.\label{eq:phit}
\end{align}
We let $\eta =\psi_i^{(j)}-\tilde{\phi}_i^{(j)}$, then we have $\pi(\eta)=0$. Therefore, $\eta\in \tilde{V}$. We see that $\psi_i^{(j)}$ and $\psi_{i,ms}^{(j)}$ satisfy
\begin{align}
a(\psi_i^{(j)},v)+\sum_{K_i^l\subset \Omega}\mu_i^{(l)}\int_{K_i^l}v\;dx
=0\quad \forall v\in H^1_0(\Omega)\label{eq:eq1}
\end{align}
and
\begin{align}
a(\psi_{i,ms}^{(j)},v)+\sum_{K_i^l\subset K_{i,k}}\mu_{i,ms}^{(l)}\int_{K_i^l}v\;dx
=0\quad \forall v\in H^1_0(K_{i,k})\label{eq:eq2}
\end{align}
for some $\mu_i^{(l)}\in Q_h$, $\mu_{i,ms}^{(l)}\in Q_h(K_{i,k})$. Subtracting the above two equations and restricting $v\in \tilde{V}_0(K_{i,k})$, we have
\begin{align*}
a(\psi_{i}^{(j)}-\psi_{i,ms}^{(j)},v)=0\quad \forall v\in \tilde{V}_0(K_{i,k}).
\end{align*}
Here, we have $\tilde{V}_0(K_{i,k})=\{v\in H^1_0(K_{i,k})|\pi(v)=0\}$. Therefore,
for $v\in \tilde{V}_0(K_{i,k})$, we can get
\begin{align*}
\|\psi_i^{(j)}-\psi_{i,ms}^{(j)}\|_a^2&
=a(\psi_i^{(j)}-\psi_{i,ms}^{(j)},\psi_i^{(j)}-\psi_{i,ms}^{(j)})\\
&=a(\psi_i^{(j)}-\psi_{i,ms}^{(j)},
\psi_i^{(j)}-\tilde{\phi}_i^{(j)}-\psi_{i,ms}^{(j)}+\tilde{\phi}_i^{(j)})
=a(\psi_i^{(j)}-\psi_{i,ms}^{(j)},\eta-v),
\end{align*}
where $-\psi_{i,ms}^{(j)}+\tilde{\phi}_i^{(j)}\in \tilde{V}_0(K_{i,k})$. Thus, we obtain
\begin{align}
\|\psi_i^{(j)}-\psi_{i,ms}^{(j)}\|_a\leq \|\eta-v\|_a.\label{eq:decay}
\end{align}
Now, we will estimate $\|\psi_i^{(j)}-\psi_{i,ms}^{(j)}\|_a$. We consider the $i$th coarse
 block $K_i$. For this block, we consider two oversampled regions $K_{i,k-1}$ and $K_{i,k}$.
Using these two overampled regions, we define the cutoff function $\chi_i^{k,k-1}$
with the properties in \eqref{cut1}-\eqref{cut2}, where we take $m=k-1$ and $M=k$.
For any coarse block $K_j\subset K_{i,k-1}$ by \eqref{cut1},
 we have $\chi_{i}^{k,k-1}\equiv 1$ on $K_j$. Since $\eta\in \tilde{V}$, we have
\begin{align*}
\sum_{K_j^n\subset K_j}\int_{K_j^n}\chi_i^{k,k-1}\eta=
\sum_{K_j^n\subset K_j}\int_{K_j^n}\eta=0.
\end{align*}

From the above result and the fact that $\chi_i^{k,k-1}\equiv 0$ in $\Omega\backslash K_{i,k}$, we have
\begin{align*}
\mbox{supp}(\pi(\chi_i^{k,k-1}\eta))\subset K_{i,k}\backslash K_{i,k-1}.
\end{align*}

By Lemma~\ref{lemma:saddle}, for the function $\pi(\chi_i^{k,k-1}\eta)$, there is $\mu\in H^1_0(\Omega)$ such that $\mbox{supp}(\mu)\subset K_{i,k}\backslash K_{i,k-1}$ and $\pi(\mu-\chi_i^{k,k-1}\eta)=0$. Moreover, it also follows from Lemma~\ref{lemma:saddle},
 the definition of $\pi$ and the Cauchy-Schwarz inequality that
\begin{align}
\|\mu\|_{a(K_{i,k}\backslash K_{i,k-1})}\leq D^{1/2}\|\kappa^{1/2}\pi(\chi_i^{k,k-1}\eta)\|_{0,K_{i,k}\backslash K_{i,k-1}}\leq D^{1/2}\|\kappa^{1/2}\chi_i^{k,k-1}\eta\|_{0,K_{i,k}\backslash K_{i,k-1}},\label{eq:muineq}
\end{align}
Hence, taking $v=\mu+\chi_i^{k,k-1}\eta$ in \eqref{eq:decay}, we can obtain
\begin{align}
\|\psi_i^{(j)}-\psi_{i,ms}^{(j)}\|_a\leq \|\eta-v\|_a\leq \|(1-\chi_i^{k,k-1})\eta\|_a+\|\mu\|_{a(K_{i,k}\backslash K_{i,k-1})}.\label{eq:decay2}
\end{align}
Next, we will estimate the two terms on the right hand side of \eqref{eq:decay2}.

Step 1: We first estimate the first term in \eqref{eq:decay2}. By a direct computation, we have
\begin{align*}
\|(1-\chi_i^{k,k-1})\eta\|_a^2\leq 2\Big(\int_{\Omega\backslash K_{i,k-1}}\kappa(1-\chi_{i}^{k,k-1})^2|\nabla \eta|^2+\int_{\Omega\backslash K_{i,k-1}}\kappa |\nabla \chi_i^{k,k-1}|^2\eta^2\Big).
\end{align*}
Note that, we have $1-\chi_i^{k,k-1}\leq 1$. For the second term on the righ hand side of the above inequality, we will use the fact that $\eta\in \tilde{V}$ and the Poincar\'{e} inequality
\begin{align*}
\|(1-\chi_i^{k,k-1})\eta\|_a^2\leq 2
(1+H^2C_{ratio})\int_{\Omega\backslash K_{i,k-1}}\kappa |\nabla \eta|^2.
\end{align*}
We will estimate the right hand side in Step 3.

Step 2: We will estimate the second term on the right hand side of \eqref{eq:decay2}. By \eqref{eq:muineq}, the fact that $|\chi_i^{k,k-1}|\leq 1$ and the Poincar\'{e} inequality, we have
\begin{align*}
\|\mu\|_{a(K_{i,k}\backslash K_{i,k-1})}^2\leq D \|\kappa^{1/2}\chi_i^{k,k-1}\eta\|_{0,K_{i,k}\backslash K_{i,k-1}}^2\leq D H^2 C_{ratio} \int_{K_{i,k}\backslash K_{i,k-1}}\kappa |\nabla \eta|^2.
\end{align*}
Combining Steps 1 and 2, we obtain
\begin{align}
\|\psi_i^{(j)}-\psi_{i,ms}^{(j)}\|_a^2\leq 2D(1+C_{ratio}H^2) \|\eta\|_{a(\Omega\backslash K_{i,k-1})}^2.\label{eq:step12}
\end{align}

Step 3: Finally, we will estimate the term $\|\eta\|_{a(\Omega\backslash K_{i,k-1})}$. We will first show that the following recursive inequality holds
\begin{align}
\|\eta\|_{a(\Omega\backslash K_{i,k-1})}\leq (1+\frac{1}{2HD^{1/2}C_{ratio}^{1/2}})^{-1}\|\eta\|_{a(\Omega\backslash K_{i,k-2})}^2,\label{eq:recursive}
\end{align}
where $k-2\geq 0$. Using \eqref{eq:recursive} in \eqref{eq:step12}, we can get
\begin{align}
\|\psi_i^{(j)}-\psi_{i,ms}^{(j)}\|_a^2\leq 2D(1+C_{ratio}H^2)(1+\frac{1}{2HD^{1/2}C_{ratio}^{1/2}})^{-1}\|\eta\|_{a(\Omega\backslash K_{i,k-2})}^2.\label{eq:step12-recursive}
\end{align}
By using \eqref{eq:recursive} again in \eqref{eq:step12-recursive}, we can obtain
\begin{align*}
\|\psi_i^{(j)}-\psi_{i,ms}^{(j)}\|_a^2&\leq 2D(1+C_{ratio}H^2)(1+\frac{1}{2HD^{1/2}C_{ratio}^{1/2}})^{1-k}\|\eta\|_{a(\Omega\backslash K_{i})}^2\\
&\leq 2D(1+C_{ratio}H^2)(1+\frac{1}{2HD^{1/2}C_{ratio}^{1/2}})^{1-k}\|\eta\|_{a}^2.
\end{align*}
By employing the definition of $\eta$, the energy minimizing property of $\psi_j^{(i)}$ and Lemma~\ref{lemma:saddle}, we have
\begin{align*}
\|\eta\|_a^2=\|\psi_{i}^{(j)}-\tilde{\phi}_i^{(j)}\|_a\leq 2\|\tilde{\phi}_i^{(j)}\|_a\leq 2D^{1/2}\|\kappa^{1/2}\delta_{lj}\|_{0,K_i}\;\forall l,j=1,\cdots,l_i.
\end{align*}
Step 4: We will prove the estimate \eqref{eq:recursive}. Let $\xi=1-\chi_i^{k-1,k-2}$. Then we see that $\xi\equiv 1$ in $\Omega\backslash K_{i,k-1}$ and $0\leq \xi\leq 1$ otherwise. Then we have
\begin{align}
\|\eta\|_{a(\Omega\backslash K_{i,k-1})}^2\leq \int_\Omega \kappa \xi^2|\nabla \eta|^2=\int_\Omega \kappa \nabla \eta\cdot \nabla (\xi^2 \eta)-2\int_\Omega \kappa \xi \eta \nabla \xi \nabla \eta.\label{eq:eta}
\end{align}
We estimate the first term in \eqref{eq:eta}. For the function $\pi(\xi^2\eta)$, using Lemma~\ref{lemma:saddle}, there exists $\gamma\in H^1_0(\Omega)$ such that $\pi(\gamma)=\pi(\xi^2\eta)$ and $\mbox{supp}(\gamma)\subset \mbox{supp}(\pi(\xi^2\eta))$. For any coarse elements $K_m\subset \Omega\backslash K_{i,k-1}$, since $\xi\equiv 1$ on $K_m$, we have for any $\phi_m^{(n)}\in Q_h(K_m)$
\begin{align*}
s_m(\xi^2\eta, \phi_m^{(n)})=0\quad \forall n=1,\ldots, l_m.
\end{align*}
On the other hand, since $\xi\equiv 0$ in $K_{i,k-2}$, we have
\begin{align*}
s_m(\xi^2\eta, \phi_m^{(n)})=0\quad \forall n=1,\ldots, l_m,\; \forall K_m\subset K_{i,k-2}.
\end{align*}
From the above two conditions, we see that $\mbox{supp}(\pi(\xi^2\eta))\subset K_{i,k-1}\backslash K_{i,k-2}$ and consequently $\mbox{supp}(\gamma)\subset K_{i,k-1} \backslash K_{i,k-2}$. Note that, since $\pi(\gamma)=\pi(\xi^2\eta)$, we have $\xi^2\eta-\gamma \in \tilde{V}$. We also note that $\mbox{supp}(\xi^2\eta-\gamma)\subset \Omega\backslash K_{i,k-2}$. By \eqref{eq:phit}, the functions $\tilde{\phi}_i^{(j)}$ and $\xi^2\eta-\gamma$ have disjoint supports, so $a(\tilde{\phi}_i^{(j)}, \xi^2\eta-\gamma)=0$. Then, by the definition of $\eta$, we have
\begin{align*}
a(\eta, \xi^2\eta-\gamma)=a(\psi_j^{(i)},\xi^2\eta-\gamma).
\end{align*}
By the construction of $\psi_i^{(j)}$, we have $a(\psi_i^{(j)},\xi^2\eta-\gamma)=0$. Then we can estimate the first term in \eqref{eq:eta} by the Cauchy-Schwarz inequality and Lemma~\ref{lemma:saddle}
\begin{align*}
\int_\Omega \kappa \nabla \eta\cdot \nabla(\xi^2\eta)&=\int_\Omega\kappa \nabla \eta\cdot \nabla \gamma\\
&\leq D^{1/2} \|\eta\|_{a(K_{i,k-1}\backslash K_{i,k-2})}\|\kappa^{1/2}\pi(\xi^2\eta)\|_{0,K_{i,k-1}\backslash K_{i,k-2}}.
\end{align*}

For all coarse elements $K\subset K_{i,k-1}\backslash K_{i,k-2}$ and assume
that $\kappa\leq \kappa_1$ within $K$, since $\pi(\eta)=0$, we have from the Poincar\'{e}
 inequality that
\begin{align*}
\|\kappa^{1/2}\pi(\xi^2\eta)\|_{0,K}^2\leq \kappa_1\|\xi^2\eta\|_{0,K}^2\leq C_{ratio} H^2\int_K \kappa |\nabla \eta|^2.
\end{align*}

Summing the above over all coarse elements $K\subset K_{i,k-1}\backslash K_{i,k-2}$, we have
\begin{align*}
\|\kappa^{1/2}\pi(\xi^2\eta)\|_{0,K_{i,k-1}\backslash K_{i,k-2}}\leq C_{ratio}^{1/2} H \|\eta\|_{a(K_{i,k-1}\backslash K_{i,k-2})}.
\end{align*}
To estimate the second term in \eqref{eq:eta}, we have from the Poincar\'{e} inequality
\begin{align*}
2\int_\Omega \kappa \xi\eta\nabla \xi\cdot\nabla \eta\leq 2 \|\kappa^{1/2} \eta\|_{0,K_{i,k-1}\backslash K_{i,k-2}}\|\eta\|_{a(K_{i,k-1}\backslash K_{i,k-2})}\leq 2HC_{ratio}^{1/2}\|\eta\|_{a(K_{i,k-1}\backslash K_{i,k-2})}^2.
\end{align*}
Hence, the preceding arguments yield the upper bound for \eqref{eq:eta}
\begin{align*}
\|\eta\|_{a(\Omega\backslash K_{i,k-1})}^2\leq 2C_{ratio}^{1/2}D^{1/2}H\|\eta\|_{a(K_{i,k-1}\backslash K_{i,k-2})}^2.
\end{align*}
Thus
\begin{align*}
\|\eta\|_{a(\Omega\backslash K_{i,k-2})}^2=\|\eta\|_{a(\Omega\backslash K_{i,k-1})}^2+\|\eta\|_{a(K_{i,k-1}\backslash K_{i,k-2})}^2\geq (1+\frac{1}{2D^{1/2}HC_{ratio}^{1/2}})\|\eta\|_{a(\Omega\backslash K_{i,k-1})}^2.
\end{align*}

\end{proof}

%
%

\begin{lemma}\label{lemma:psi}
With the same assumptions as in Lemma~\ref{lemma:decay}, we can obtain
\begin{align*}
\|\sum_{i=1}^N(\psi_i^{(j)}-\psi_{i,ms}^{(j)})\|_a^2&\leq C(k+1)^2\sum_{i=1}^N \|\psi_i^{(j)}-\psi_{i,ms}^{(j)}\|_a^2.
\end{align*}

\end{lemma}

\begin{proof}
Let $w=\sum_{i=1}^N (\psi_i^{(j)}-\psi_{i,ms}^{(j)})$. By the constructions in \eqref{multiscale1}-\eqref{multiscale2} and \eqref{global-multiscale} and Lemma~\ref{lemma:saddle}, there is $z_i\in H^1_0(\Omega)$ such that
\begin{align*}
\pi(z_i)=\pi((1-\chi_{i}^{k+1,k})w), \quad \mbox{supp}(z_i)\subset K_{i,k+1}\backslash K_{i,k},\quad \|z_i\|_a\leq D\|\kappa^{1/2}\pi((1-\chi_{i}^{k+1,k}w))\|_0.
\end{align*}
It then follows from \eqref{eq:eq1} and \eqref{eq:eq2} that
\begin{align}
a(\psi_i^{(j)}-\psi_{i,ms}^{(j)},v)+
\sum_{K_i^l\subset K_{i,k}}(\mu_i^{(l)}-\mu_{i,ms}^{(l)})\int_{K_i^l}v\;dx=0\quad \forall v\in H^1_0(K_{i,k}).\label{eq:error}
\end{align}
Putting $v=((1-\chi_i^{k+1,k})w)-z_i$ in \eqref{eq:error}, we can obtain
\begin{align*}
a(\psi_i^{(j)}-\psi_{i,ms}^{(j)},((1-\chi_i^{k+1,k})w)-z_i)=0.
\end{align*}
Thus
\begin{align}
\|\sum_{i=1}^N(\psi_i^{(j)}-\psi_{i,ms}^{(j)})\|_a^2
=a(w,w)=\sum_{i=1}^Na(\psi_i^{(j)}-\psi_{i,ms}^{(j)},w)
=\sum_{i=1}^Na(\psi_i^{(j)}-\psi_{i,ms}^{(j)},\chi_i^{k+1,k}w+z_i).\label{eq:psi}
\end{align}
For each $i=1,2,\ldots,N$, we have
\begin{align*}
\|\chi_i^{k+1,k}w\|_a^2\leq C (\|w\|_{a(K_{i,k+1})}^2+\|\kappa^{1/2}w\|_{0,K_{i,k+1}}^2)\leq (1+C_{ratio}H^2) \|w\|_{a(K_{i,k+1})}^2.
\end{align*}
In addition, since $\pi_{mn}(w)=0$ for all $K_m^n$ with $m\neq i,\forall n=1,\cdots,l_m$, we can get
\begin{align*}
\|z_i\|_a^2&\leq D^2 \|\kappa^{1/2}\pi((1-\chi_i^{k+1,k})w)\|_0^2\leq D^2\|\kappa^{1/2}\pi(\chi_i^{k+1,k}w)\|_{0,K_{i,k+1}}^2\leq D^2\|\kappa^{1/2}w\|_{0,K_{i,k+1}}^2\\
&\leq D^2C_{ratio}H^2\|w\|_{a(K_{i,k+1})}^2.
\end{align*}
which yields the desired estimate by combining with \eqref{eq:psi}.


\end{proof}

The convergence of the multiscale solution can be stated in the next theorem.
\begin{theorem}
Let $u$ be the solution of \eqref{eq:weak} and $u_h$ be the multiscale solution, then we have
\begin{align}
\|u-u_{ms}\|_a\leq C HC_{ratio}^{1/2}\|\kappa^{-1/2}f\|_0
+C (1+k)E^{1/2}C_{ratio}^{1/2}\|\kappa^{1/2}u_{glo}\|_0.\label{eq:convergence1}
\end{align}
Moreover, if $k=O(log(\frac{max\{\kappa\}}{H}))$, then we have
\begin{align}
\|u-u_{ms}\|_a&\leq C H C_{ratio}^{1/2}\|\kappa^{-1/2}f\|_0,\label{eq:convergence2}\\
\|u-u_{ms}\|_0&\leq C H^2C_{ratio}^{1/2}\kappa_{min}^{-1/2} \|\kappa^{-1/2}f\|_0.\nonumber
\end{align}

\end{theorem}

\begin{proof}

We write $u_{glo}=\sum_{i=1}^N\sum_{j=1}^{l_i}c_i^{(j)}\psi_i^{(j)}$. Then we define $v=\sum_{i=1}^N\sum_{j=1}^{l_i}c_i^{(j)}\psi_{i,ms}^{(j)}$. It then follows from the Galerkin orthogonality that
\begin{align}
\|u-u_{ms}\|_a\leq \|u-v\|_a\leq \|u-u_{glo}\|_a
+\|\sum_{i=1}^N\sum_{j=1}^{l_i}c_i^{(j)}(\psi_i^{(j)}-\psi_{i,ms}^{(j)})\|_a.\label{eq:ms}
\end{align}
Lemma~\ref{lemma:psi} yields
\begin{align*}
\|\sum_{i=1}^N\sum_{j=1}^{l_i}c_i^{(j)}(\psi_i^{(j)}-\psi_{i,ms}^{(j)})\|_a^2&\leq C (1+k)^2\sum_{i=1}^N\|\sum_{j=1}^{l_i}c_i^{(j)}(\psi_i^{(j)}-\psi_{i,ms}^{(j)})\|_a^2\\
&\leq C (k+1)^2C_{ratio}\sum_{i=1}^N \|\kappa^{1/2}\sum_{j=1}^{l_i}c_i^{(j)}\delta_{ij}\|_0^2
\leq C (k+1)^2C_{ratio}\|\kappa^{1/2}u_{glo}\|_0^2.
\end{align*}
The above equation together with Lemma~\ref{lemma:conglo} and \eqref{eq:ms} implies
\begin{align*}
\|u-u_{ms}\|_a\leq C
\Big(HC_{ratio}^{1/2}\|\kappa^{-1/2}f\|_0+(1+k)E^{1/2}C_{ratio}^{1/2}\|\kappa^{1/2}u_{glo}\|_0\Big).
\end{align*}
This yields \eqref{eq:convergence1}.

The Poincar\'{e} inequality yields
\begin{align*}
\|\kappa^{1/2}u_{glo}\|_0^2\leq \kappa_{min}^{-1}\kappa_{max}\|u_{glo}\|_a^2.
\end{align*}
An application of \eqref{eq:global} and the Cauchy-Schwarz inequality gives
\begin{align*}
\|u_{glo}\|_a^2=\int_{\Omega} f u_{glo}\leq C \|\kappa^{-1/2}f\|_0\|\kappa^{1/2}u_{glo}\|_0.
\end{align*}
Therefore
\begin{align*}
\|\kappa^{1/2}u_{glo}\|_0\leq \kappa_{min}^{-1}\kappa_{max} \|\kappa^{-1/2}f\|_0.
\end{align*}

Then proceeding analogously to \cite{ChungEfendievleung18}
and employing the fact that $C_{ratio}$ is relatively small, we can conclude that if $k=O(log(\frac{max\{\kappa\}}{H}))$,
then we can obtain \eqref{eq:convergence2}.

Next, we consider the estimate for $\|u-u_{glo}\|_0$. Consider the dual problem
\begin{align*}
a(z,v)=(u-u_{ms},v) \quad \forall v\in H^1_0(\Omega).
\end{align*}
Then, the Cauchy-Schwarz inequality and \eqref{eq:convergence2} yield
\begin{align*}
\|u-u_{ms}\|_0^2&=a(z,u-u_{ms})=a(z-z_{ms},u-u_{ms})\leq  \|z-z_{ms}\|_a\|u-u_{ms}\|_a\\
&\leq C HC_{ratio}^{1/2}\kappa_{min}^{-1/2}\|u-u_{ms}\|_0\|u-u_{ms}\|_a.
\end{align*}
Thus
\begin{align*}
\|u-u_{ms}\|_0\leq C HC_{ratio}^{1/2}\kappa_{min}^{-1/2}\|u-u_{ms}\|_a.
\end{align*}


\end{proof}

\section{Numerical experiments}\label{sec:numerical}

%

This section presents numerical experiments to verify the capability of the proposed method to the problem with high contrast medium. To compare the results, we exploit the relative $L^2$ error between coarse cell average of the fine-scale solution $\bar{u}_f$ and the upscaled coarse grid solution $\bar{u}$
\begin{align*}
e_{L^2}=\|\bar{u}_f-\bar{u}\|_{L^2},\quad \|\bar{u}_f-\bar{u}\|_{L^2}=\frac{\sum_{K}\int_{K}(\bar{u}_f-\bar{u}^K)^2\;dx}{\sum_K \int(\bar{u}_f)^2\;dx},\quad \bar{u}_f^K=\frac{1}{|K|}\int_K u_f\;dx.
\end{align*}

\begin{example}\label{ex1}
\end{example}
In this example, we take $\Omega=(0,1)^2$, $u=0$ on $\partial \Omega$ and we set $f=1$. The medium $\kappa$ is shown in Figure~\ref{a} and we assume that the fine mesh size $h$ to be $\sqrt{2}/400$, That is, the medium $\kappa$ has a $400\times 400\times 2$ resolution. We consider the contrast of the medium is $10^4$ where the value of $\kappa$ is large in the yellow region. For the NLMC method, we consider two continua. 

\begin{figure}[H]
\centering
\includegraphics[width=7cm]{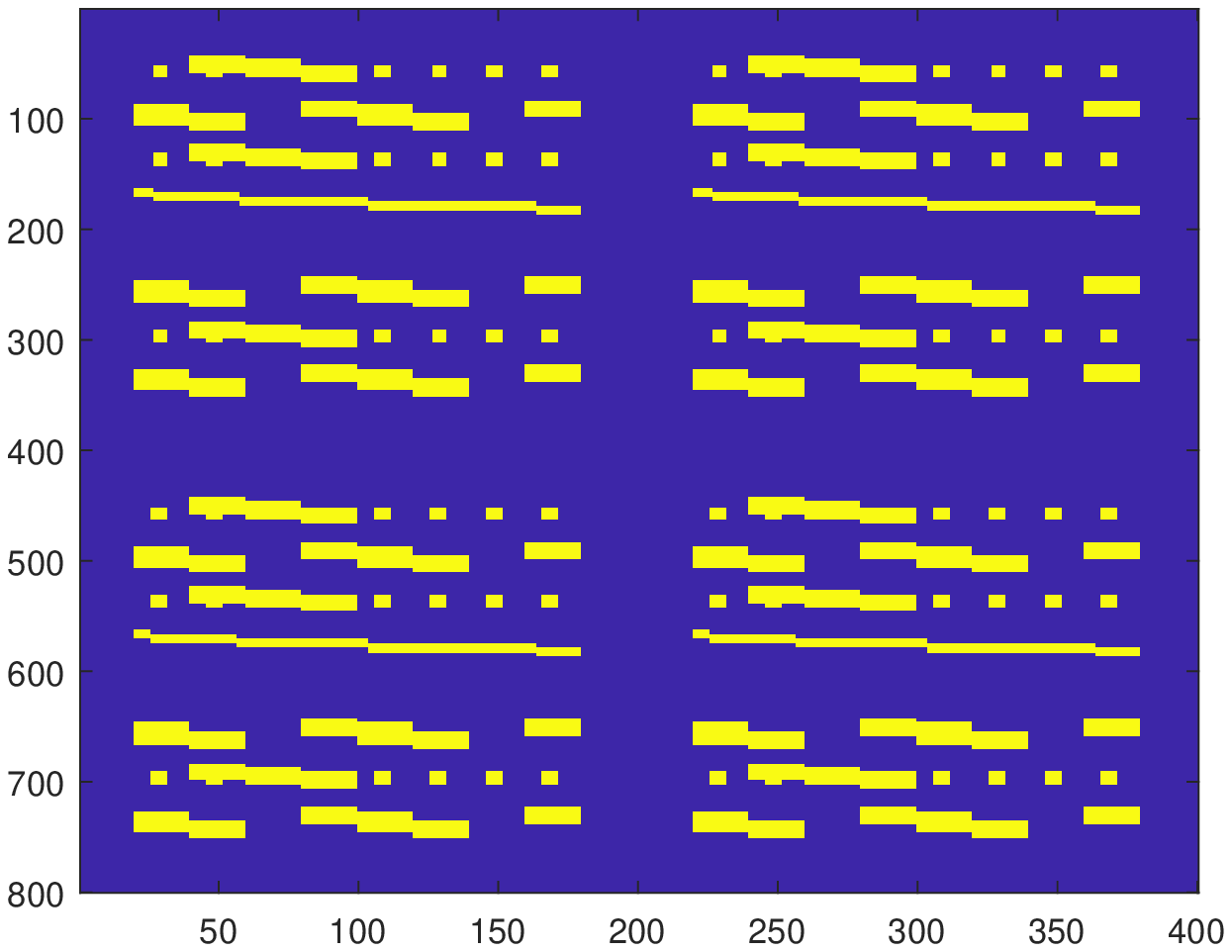}
\caption{The medium $\kappa$ for Example~\ref{ex1}.}
\label{a}
\end{figure}
The fine scale and upscaled solutions for coarse mesh $20\times 20$ with $4$ oversampling layers can be found in Figure~\ref{fine1}-Figure~\ref{coarse1}. In Figure~\ref{fine1}, we display the downscale and fine scale solution and in Figure~\ref{coarse1} we show the upscaled coarse solution and the average value of the fine scale solution. In addition, the numerical results for $40\times 40$ coarse mesh with 5 oversampling layers are reported in Figure~\ref{solution40}-Figure~\ref{average40}. From which we observe very good agreement between
the fine-scale solution and the computed upscaled solution.
\begin{figure}[H]
\centering
\scalebox{0.3}{
\includegraphics[width=20cm]{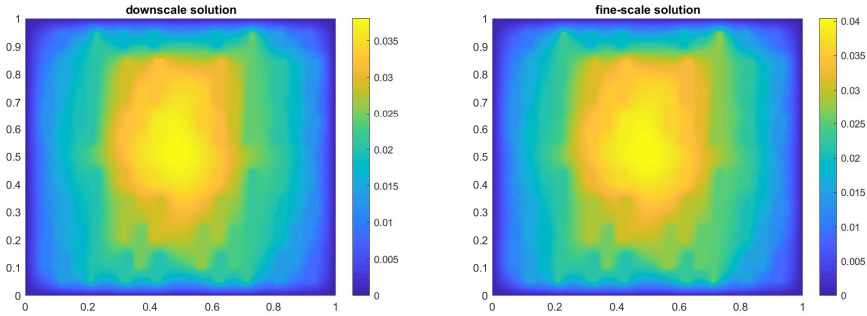}
}
\scalebox{0.3}{
\includegraphics[width=20cm]{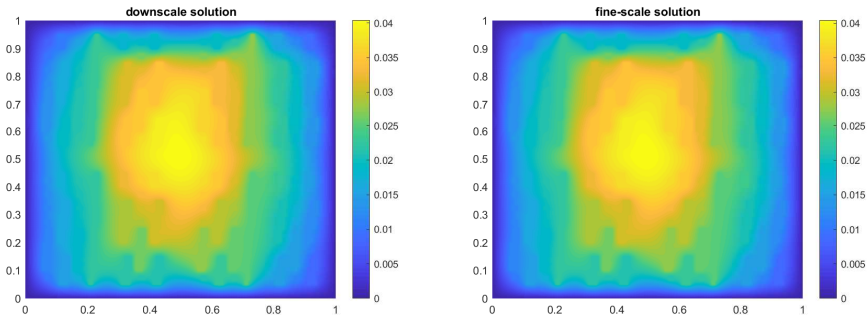}
}
\caption{Downscale solution and fine-scale solution for Example~\ref{ex1}.}
\label{fine1}
\end{figure}

\begin{figure}[H]
\centering
\scalebox{0.3}{
\includegraphics[width=20cm]{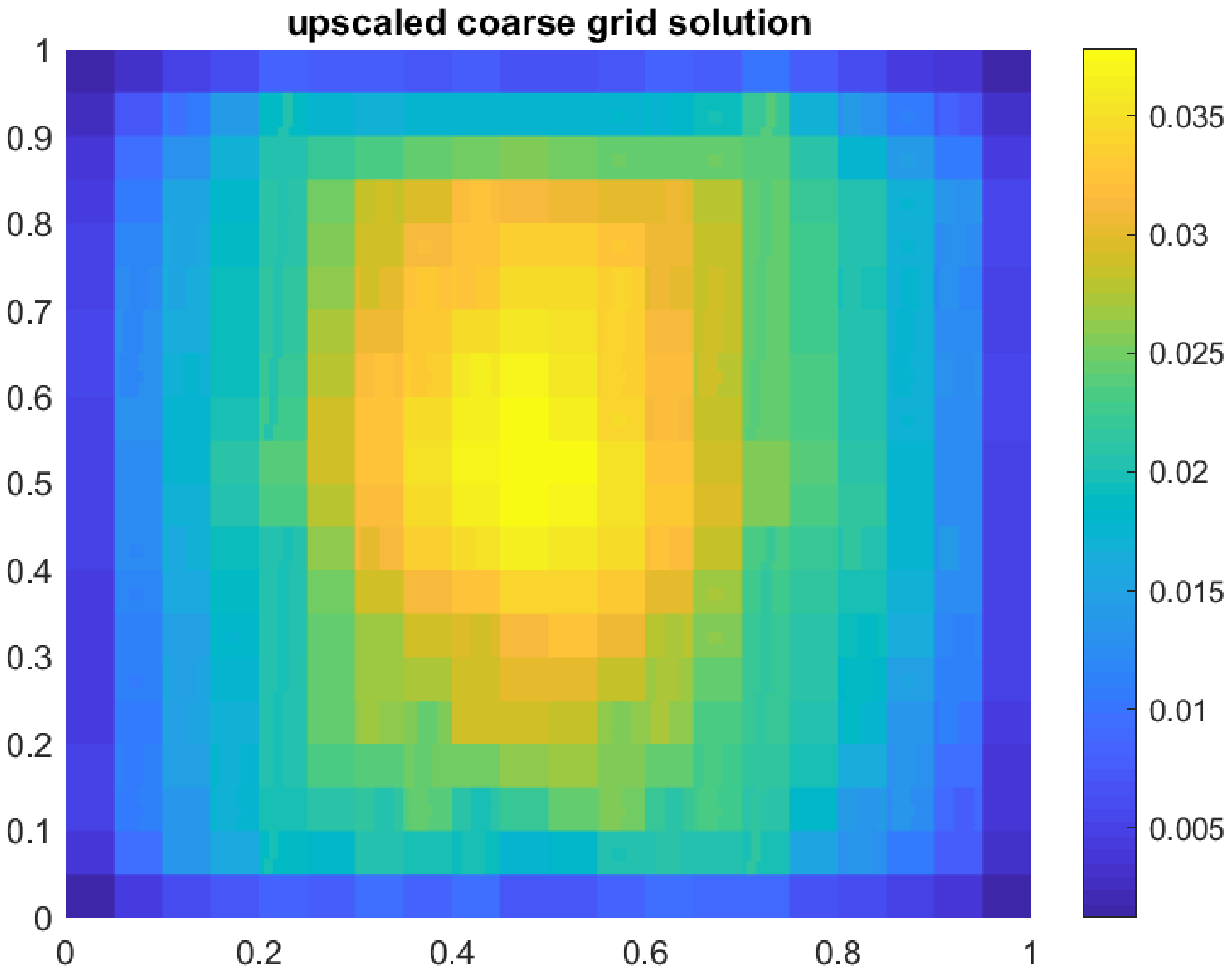}
}
\scalebox{0.3}{
\includegraphics[width=20cm]{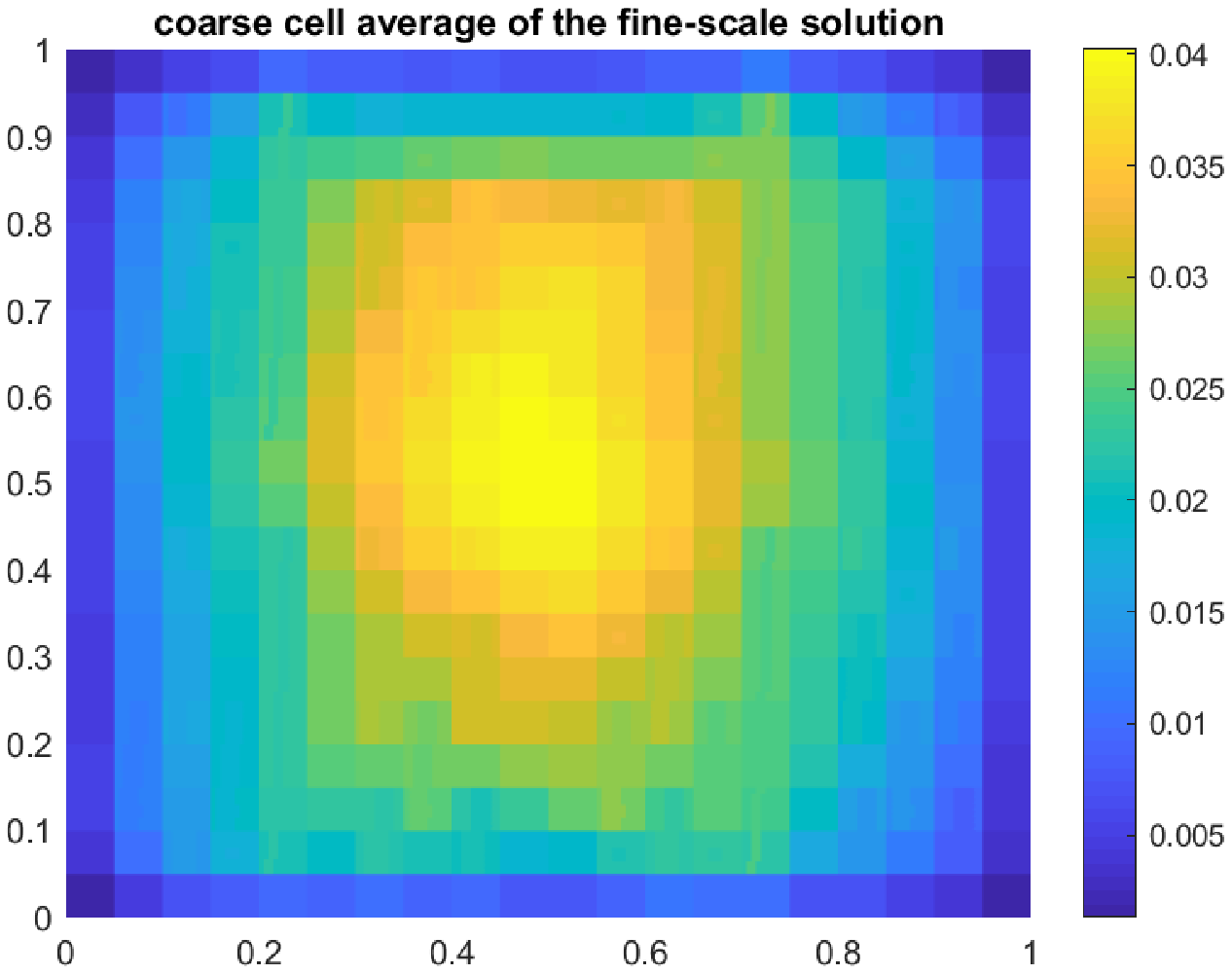}
}
\caption{Coarse scale solution and coarse cell average of fine-scale solution for Example~\ref{ex1}.}
\label{coarse1}
\end{figure}

\begin{figure}[H]
\centering
\scalebox{0.3}{
\includegraphics[width=20cm]{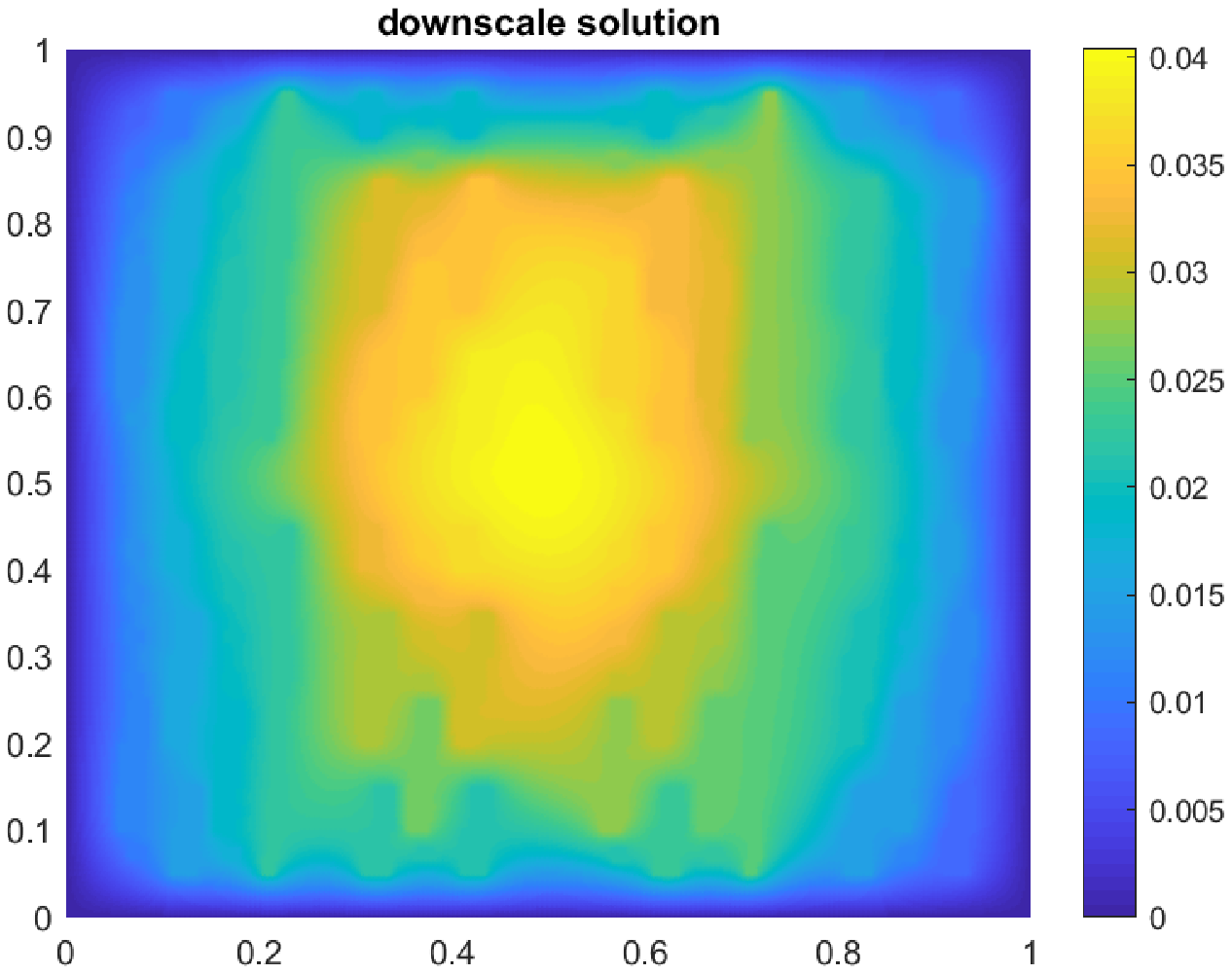}
}
\scalebox{0.3}{
\includegraphics[width=20cm]{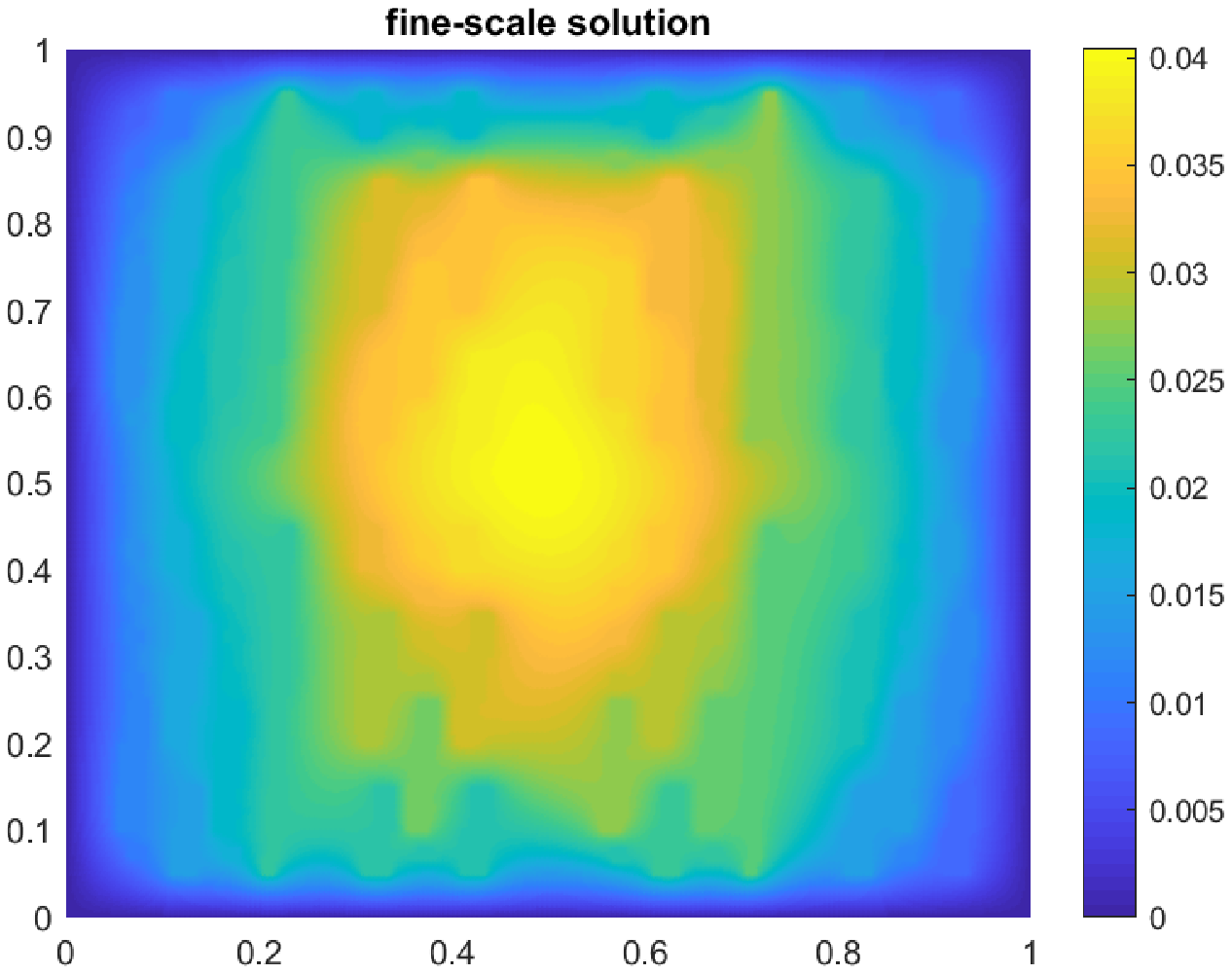}
}
\caption{Downscale solution and fine-scale solution for Example~\ref{ex1}.}
\label{solution40}
\end{figure}

\begin{figure}[H]
\centering
\scalebox{0.3}{
\includegraphics[width=20cm]{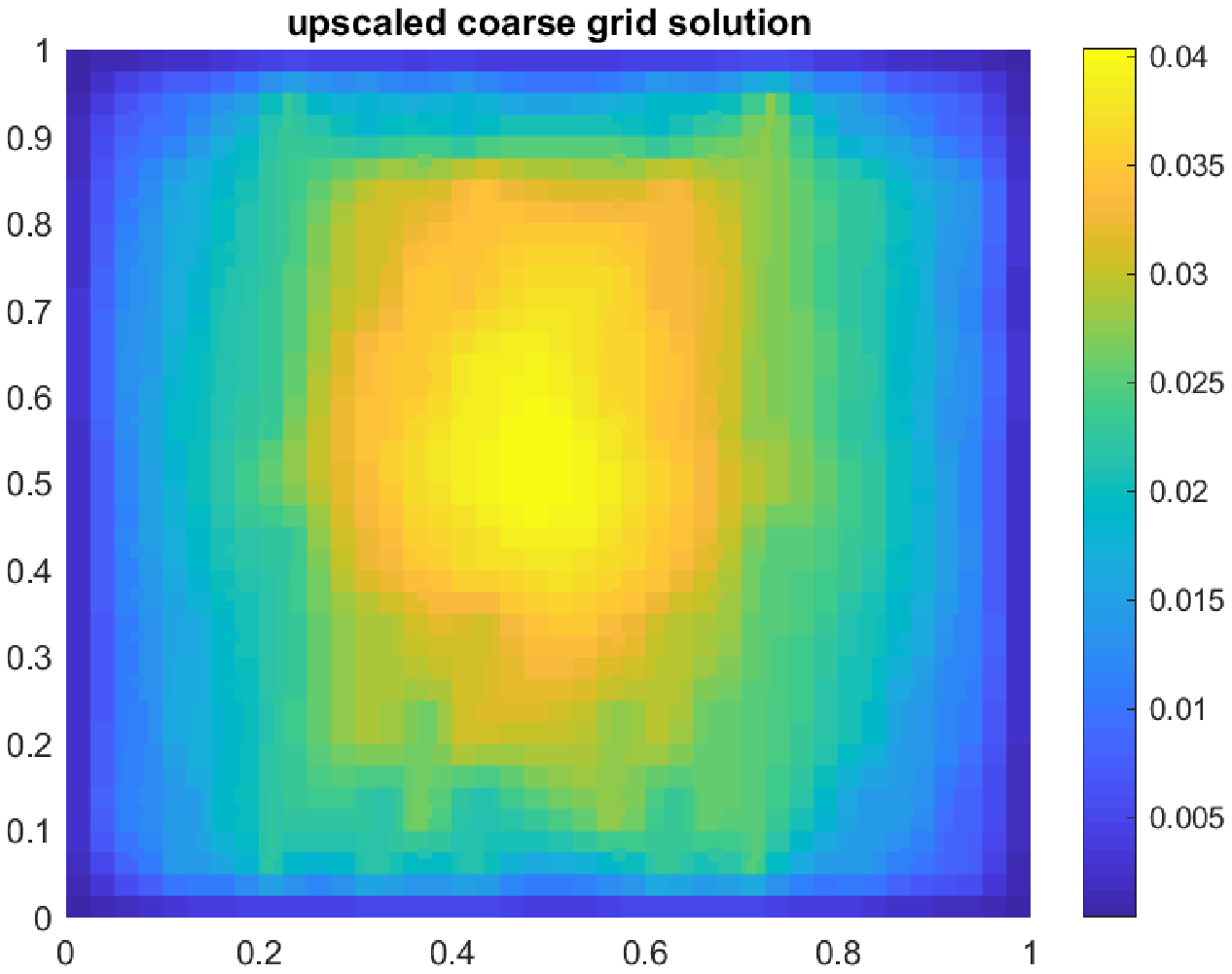}
}
\scalebox{0.3}{
\includegraphics[width=20cm]{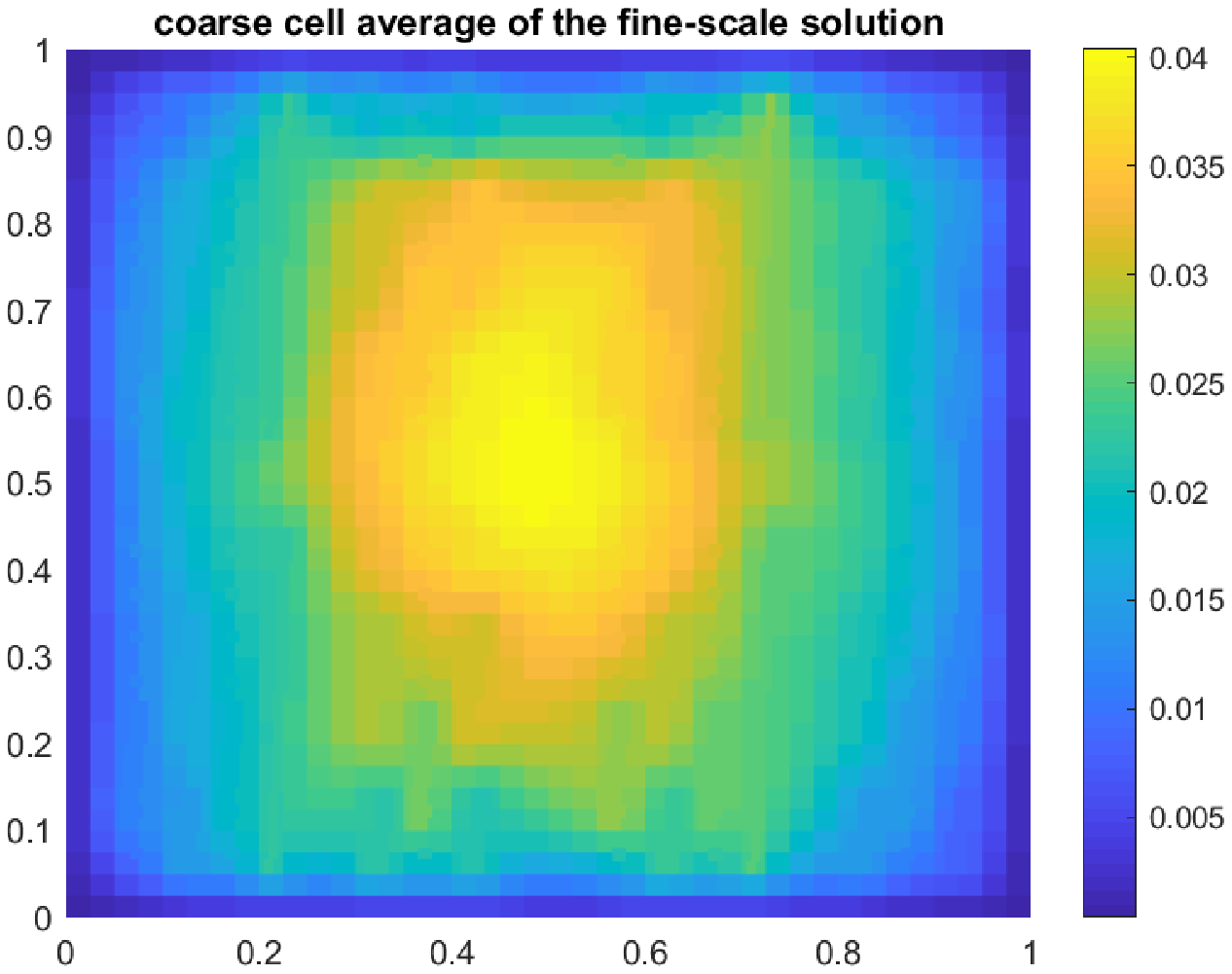}
}
\caption{Coarse scale solution and coarse cell average of fine-scale solution for Example~\ref{ex1}.}
\label{average40}
\end{figure}

In Table~\ref{ex1-con}, we present the relative $L^2$ error with varying coarse grid size. With proper choices of oversampling layers, we can see that the error converges. The relative $L^2$ error for coarse grids $20\times 20$ and $40\times 40$, and for different number of oversampling layers are reported in Table~\ref{mesh-layer}. From which we can see that for a fixed contrast value, the error decays as the oversampling size increases. In addition, as the number of coarse grid increases, more oversampling layers are required in order to achieve the desired error. Furthermore, for a fixed oversampling size, the performance of the scheme will deteriorate as the medium contrast increases, which can be illustrated by Table~\ref{layer-ratio}.
\begin{table}[H]
\centering
\begin{tabular}{l|llll}
\hline
$H$& oversampling coarse layers & $e_{L^2}$ \\
\hline
$\frac{1}{10}$& 3  & 0.1678\\
$\frac{1}{20}$& 4  & 0.0808\\
$\frac{1}{40}$& 5 & 0.0453\\
\hline
\end{tabular}
\caption{Relative $L^2$ error for Example~\ref{ex1} with varying coarse grid size.}
\label{ex1-con}
\end{table}

\begin{table}[H]
\centering
\begin{tabular}{l|ll}
\hline
Layer& coarse mesh \;20$\times$ 20 &coarse mesh \;40$\times$40\\
\hline
1 & 0.9690    & 0.9876 \\
3 & 0.4816    & 0.9136 \\
4 & 0.0808    & 0.4772 \\
5 & 0.0054    & 0.0453 \\
6 & 2.759e-4  & 0.0012  \\
\hline
\end{tabular}
\caption{Relative $L^2$ error with respect to different number of oversampling layers
for Example~\ref{ex1}.}
\label{mesh-layer}
\end{table}

\begin{table}[H]
\centering
\begin{tabular}{l|llll}
\hline
Layer $\backslash$ Contrast& $10^3$ & $10^4$ & $10^5$ & $10^6$\\
\hline
3 & 0.1575    & 0.4816 &0.6319 &0.6526\\
4 & 0.0103    & 0.0808 &0.3796 &0.6081\\
5 & 6.0346e-4 & 0.0054 &0.0496 &0.2943\\
\hline
\end{tabular}
\caption{Comparison of various number of oversampling layers and different contrast values
for Example~\ref{ex1}.}
\label{layer-ratio}
\end{table}

\begin{example}\label{ex3}

\end{example}

In this example, we again take $\Omega=(0,1)^2$ and the profile of $\kappa$ is shown in Figure~\ref{ex3-a}, where $\kappa$ is taken to be some random numbers between $(1,10)$ for the blue region and $\kappa$ is $10^3$ or $10^4$ in the yellow region. For the NLMC method, we consider three continua,
namely, $\{ 1\leq \kappa \leq 10\}$, $\{ \kappa \approx 10^3\}$ and $\{ \kappa \approx 10^4\}$.
In addition, $f$ is taken to be
\begin{align*}f(x,y)=
  \begin{cases}
    1 \quad \forall\; 0\leq x\leq 0.1, 0\leq y\leq 0.1,\\
    0 \quad\mbox{otherwise}\\
  \end{cases}
\end{align*}

\begin{figure}[H]
\centering
\includegraphics[width=7cm]{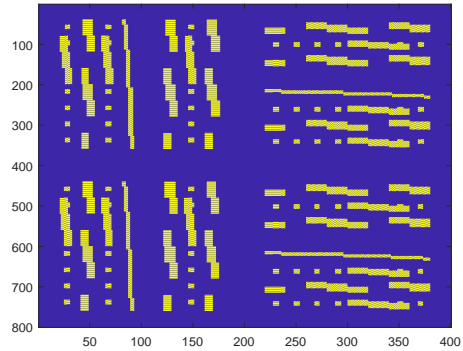}
\caption{The medium $\kappa$ for Example~\ref{ex3}.}
\label{ex3-a}
\end{figure}

The fine scale and upscaled solutions for coarse mesh $20\times 20$ with $4$ oversampling layers can be found in Figure~\ref{ex3-solution-fine20}-Figure~\ref{ex3-solution-coarse20}. In Figure~\ref{ex3-solution-fine20}, we display the downscale and fine scale solution and in Figure~\ref{ex3-solution-coarse20} we show the upscaled coarse solution and the average value of the fine scale solution. The numerical results for $40\times 40$ coarse mesh with 5 oversampling layers are reported in Figure~\ref{ex3-solution-fine40}-Figure~\ref{ex3-solution-coarse40}. We can observe that the fine-grid solution and the upscaled coarse grid solution match well.
\begin{figure}[H]
\centering
\scalebox{0.3}{
\includegraphics[width=20cm]{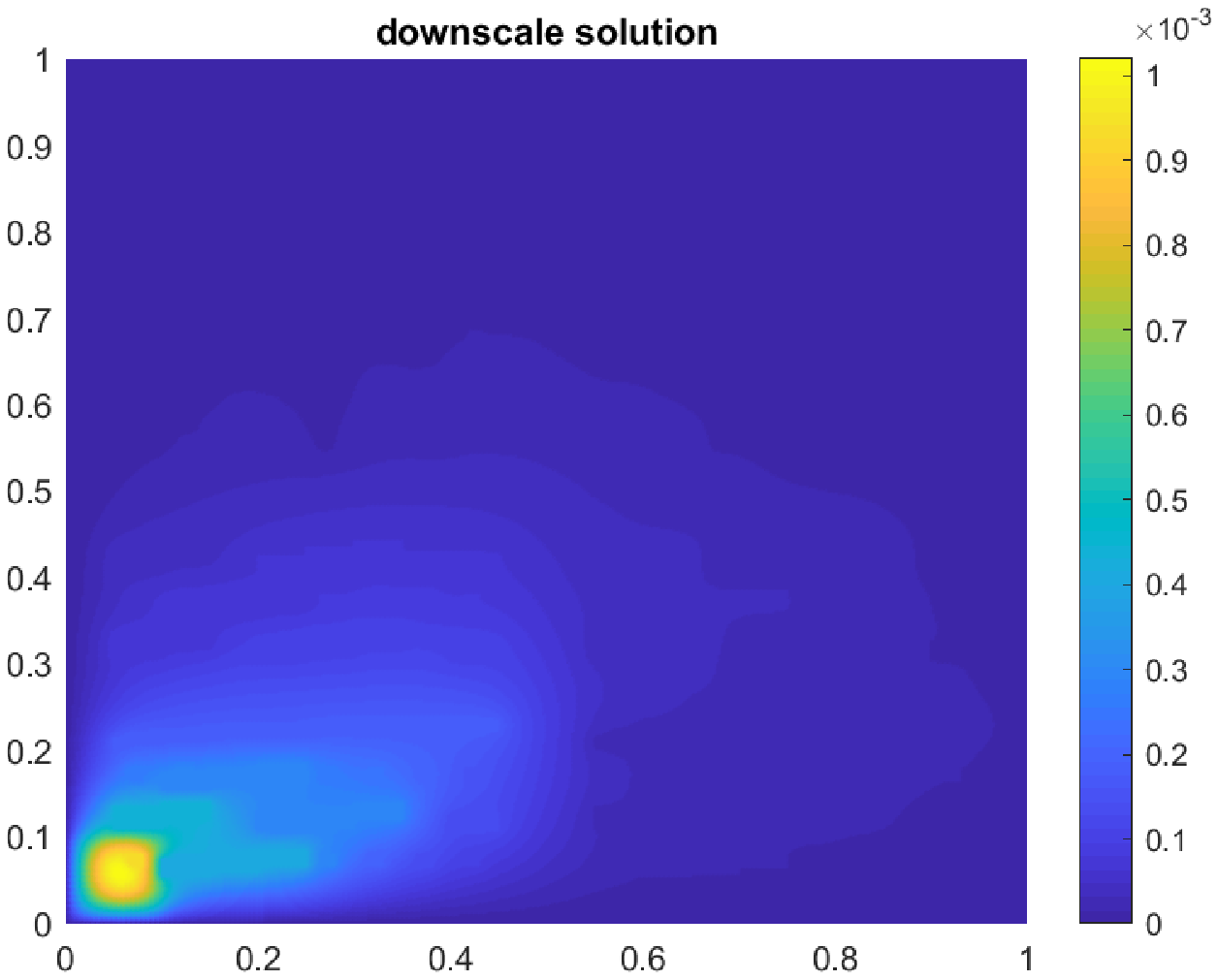}
}
\scalebox{0.3}{
\includegraphics[width=20cm]{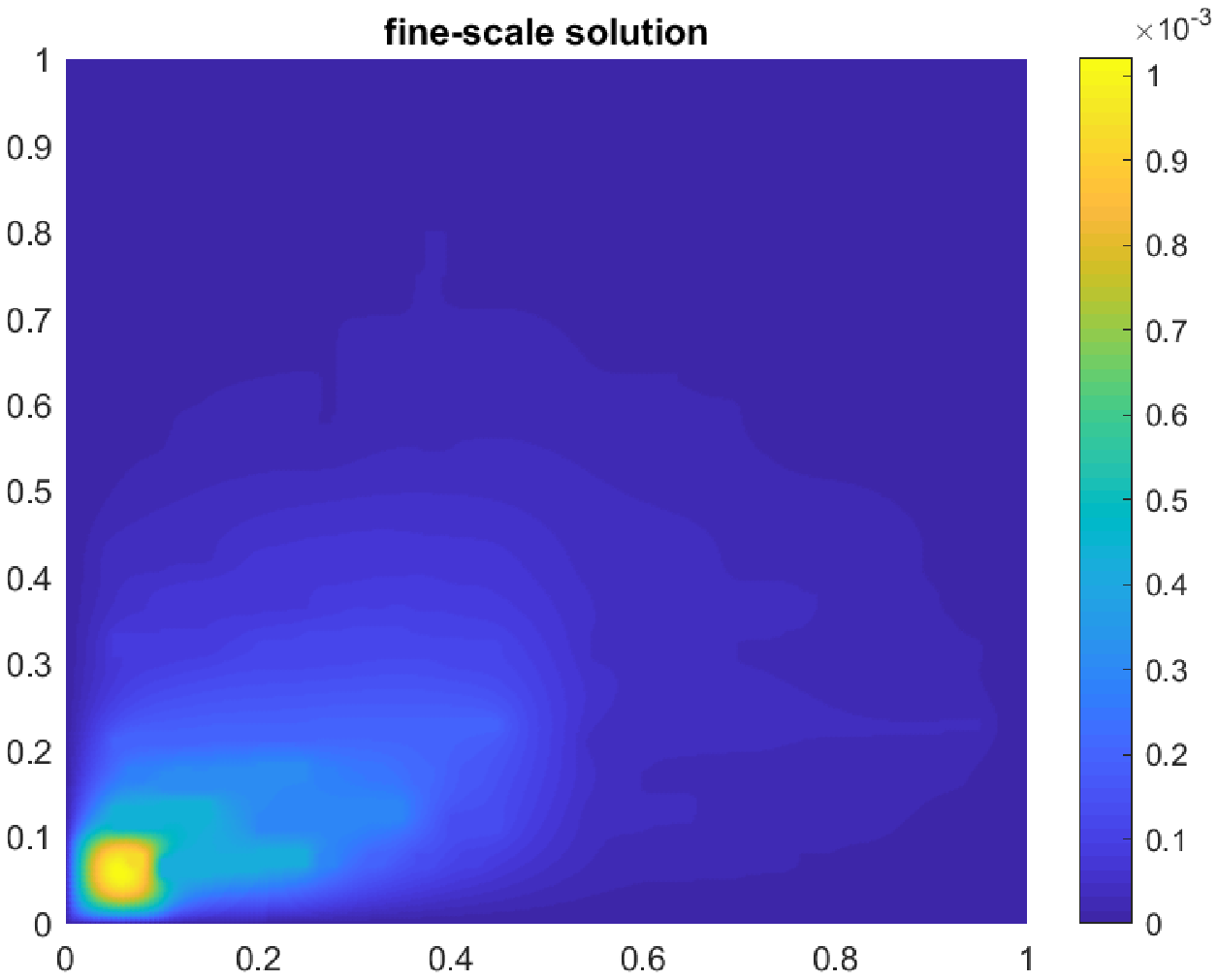}
}
\caption{Downscale solution and fine-scale solution.}
\label{ex3-solution-fine20}
\end{figure}

\begin{figure}[H]
\centering
\scalebox{0.3}{
\includegraphics[width=20cm]{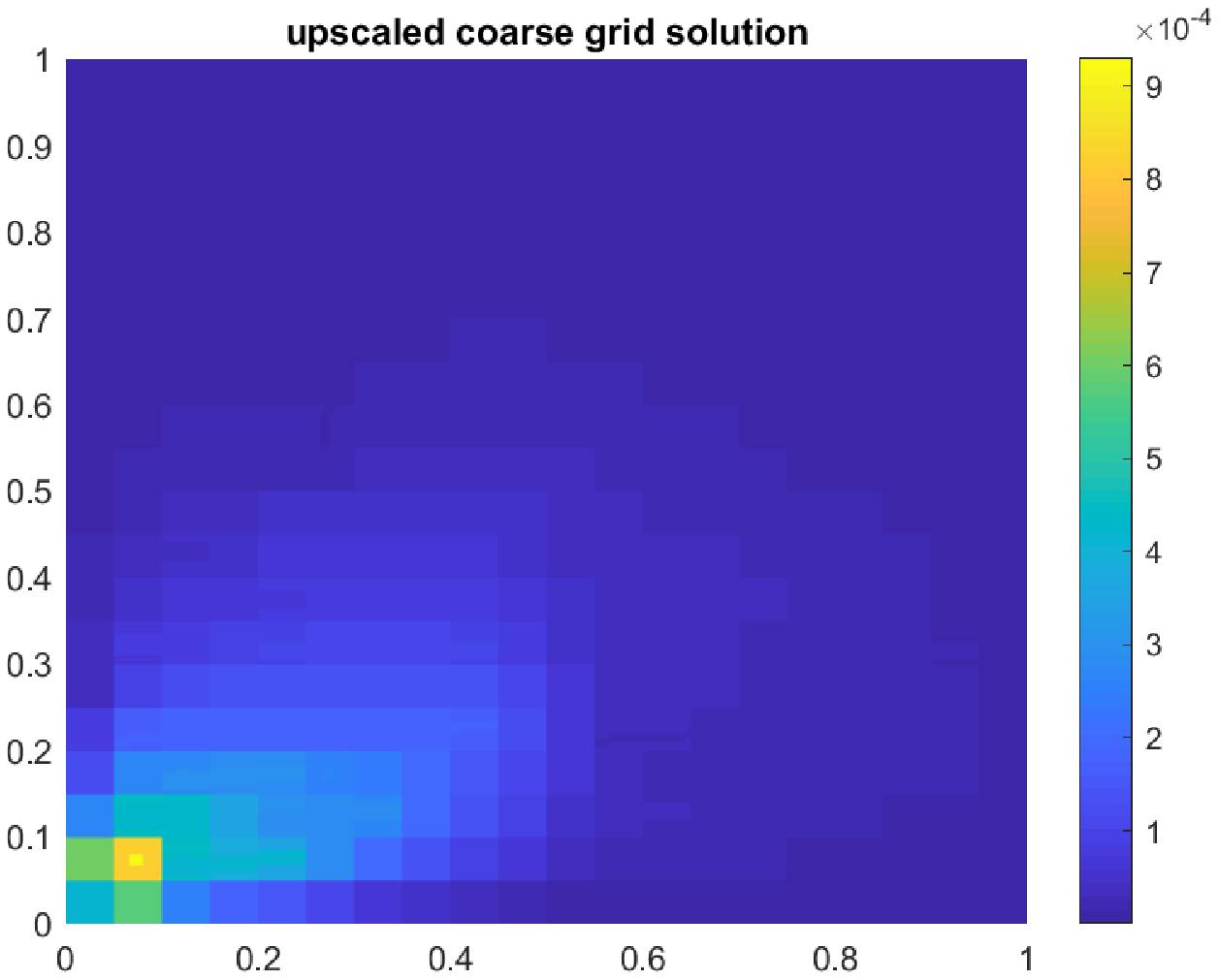}
}
\scalebox{0.3}{
\includegraphics[width=20cm]{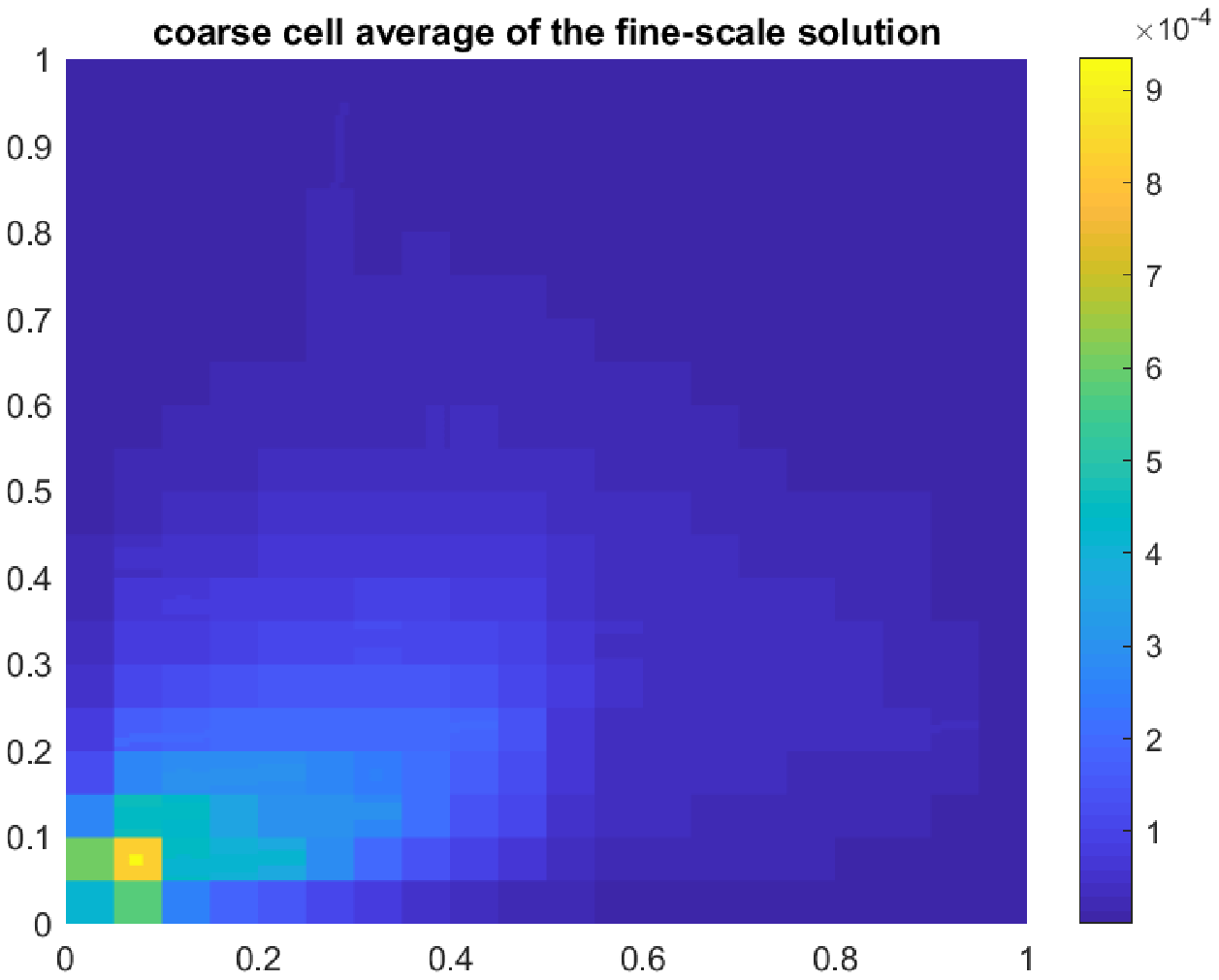}
}
\caption{coarse scale solution and coarse cell average of fine-scale solution.}
\label{ex3-solution-coarse20}
\end{figure}

\begin{figure}[H]
\centering
\scalebox{0.3}{
\includegraphics[width=20cm]{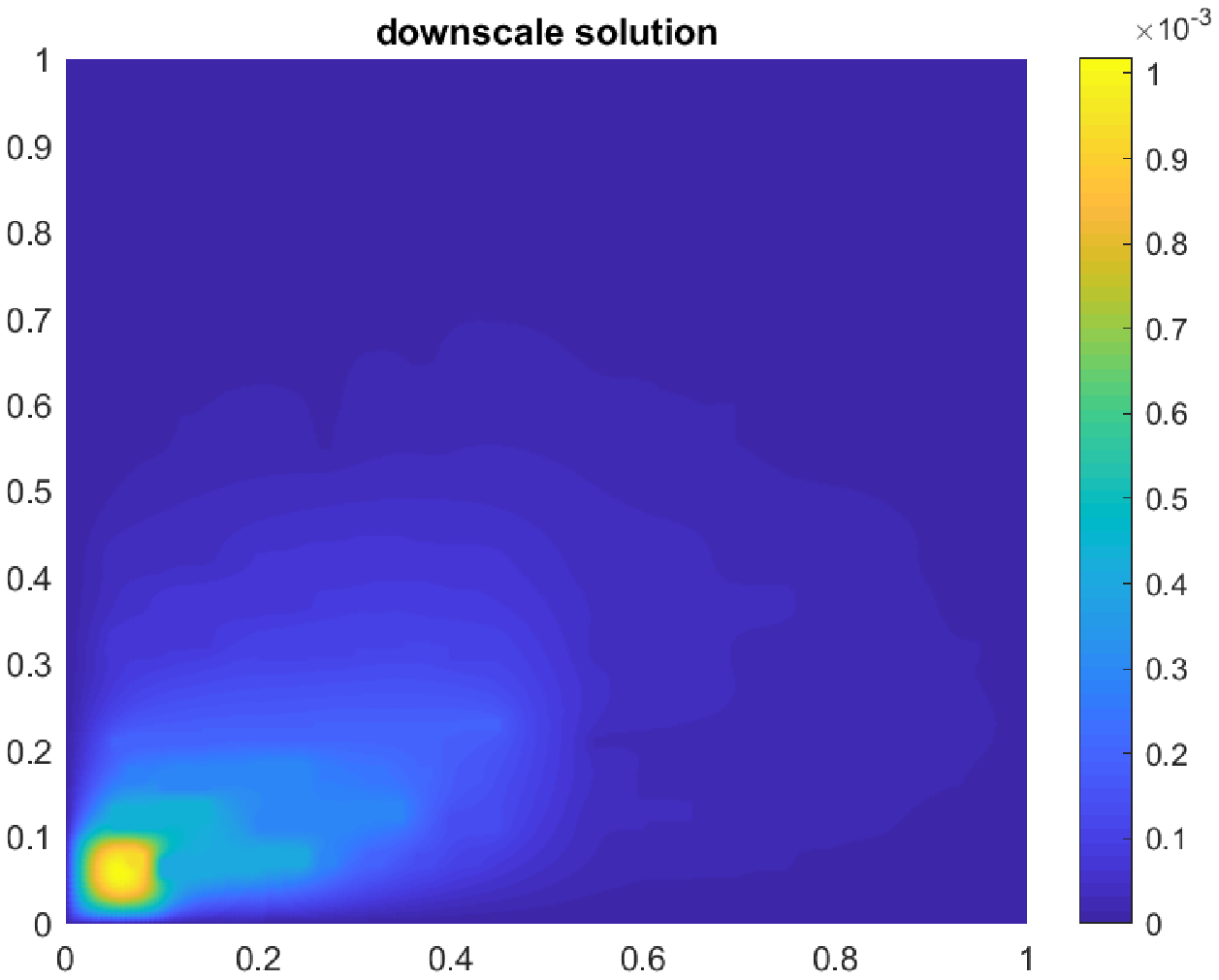}
}
\scalebox{0.3}{
\includegraphics[width=20cm]{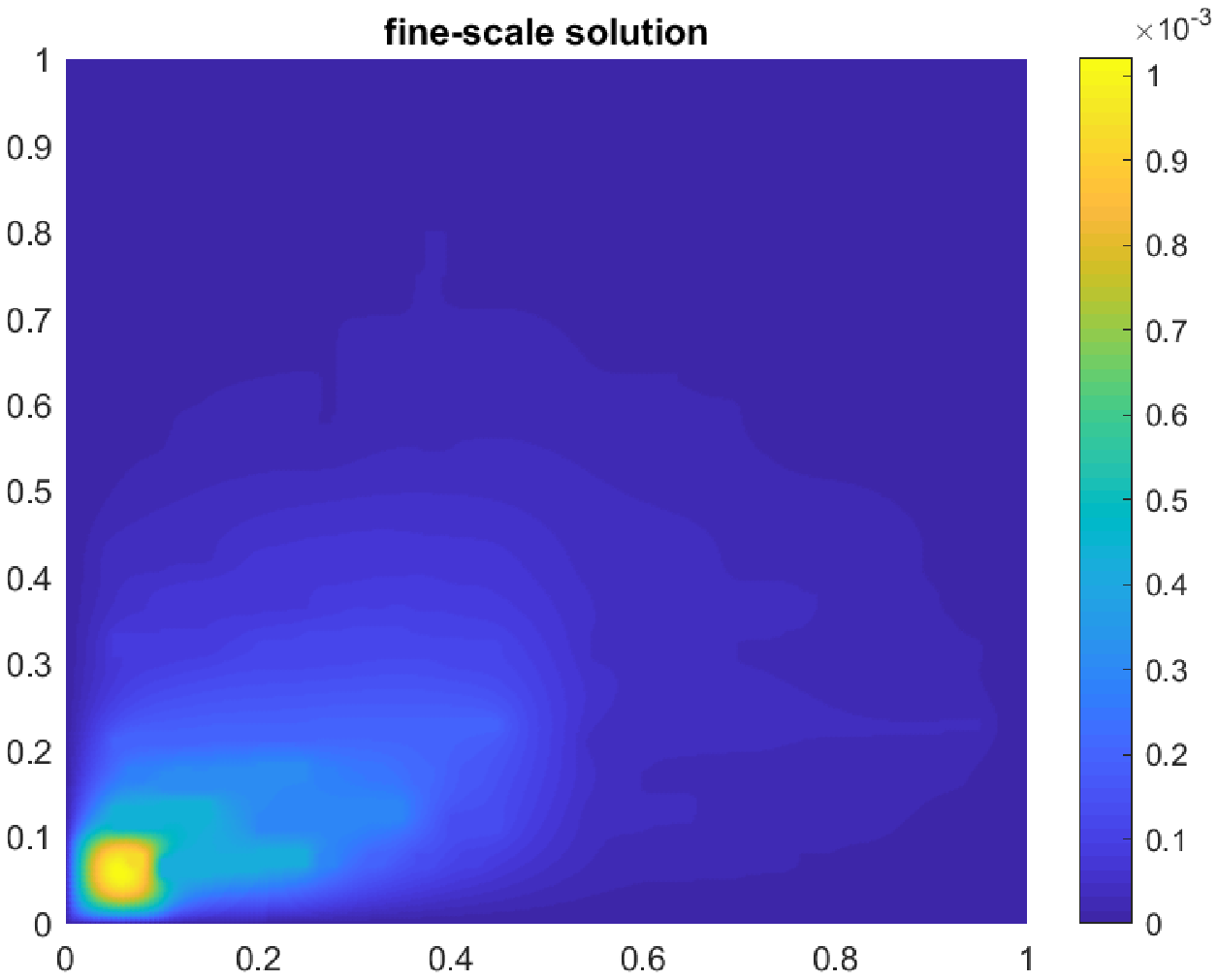}
}
\caption{Downscale solution and fine-scale solution.}
\label{ex3-solution-fine40}
\end{figure}

\begin{figure}[H]
\centering
\scalebox{0.3}{
\includegraphics[width=20cm]{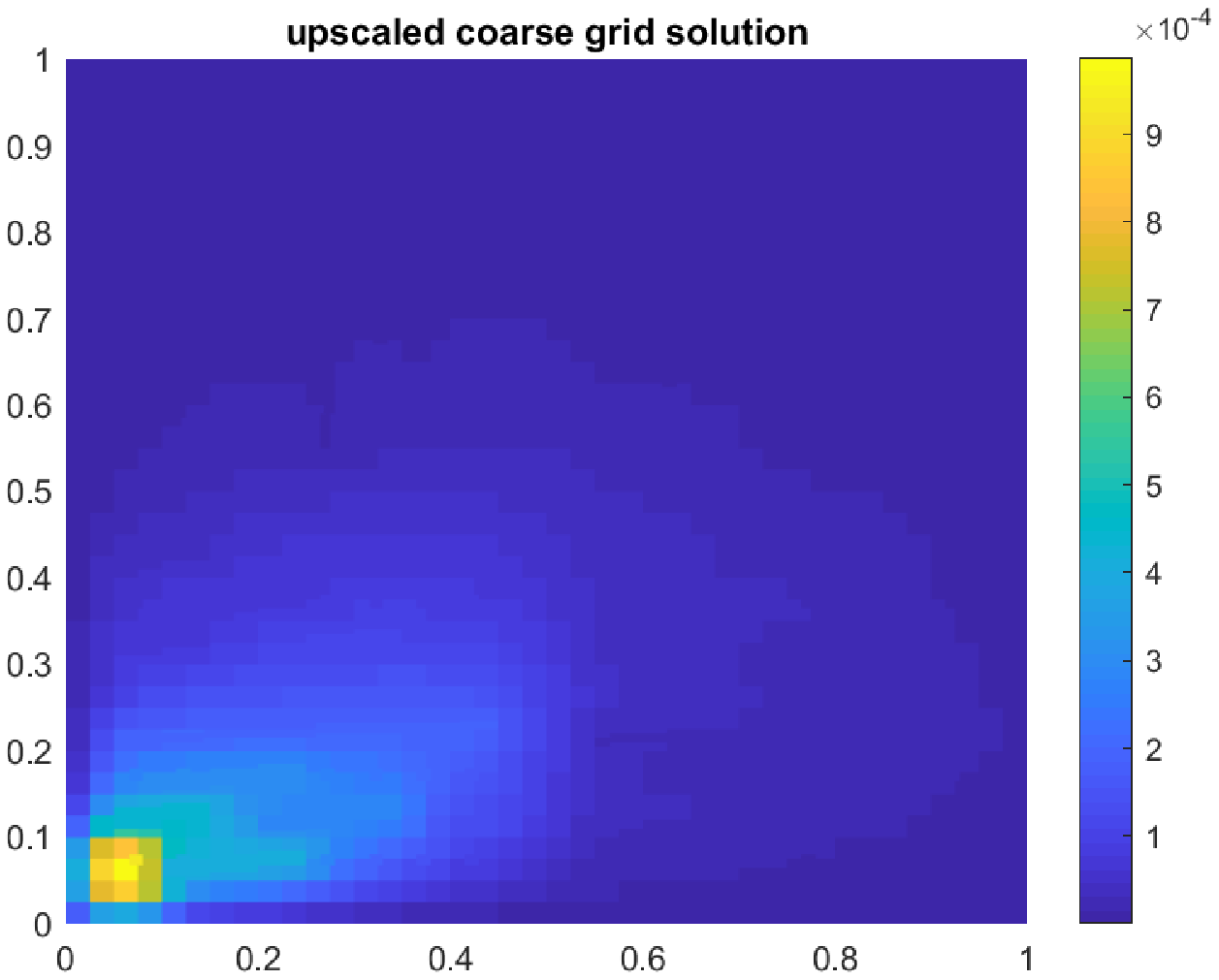}
}
\scalebox{0.3}{
\includegraphics[width=20cm]{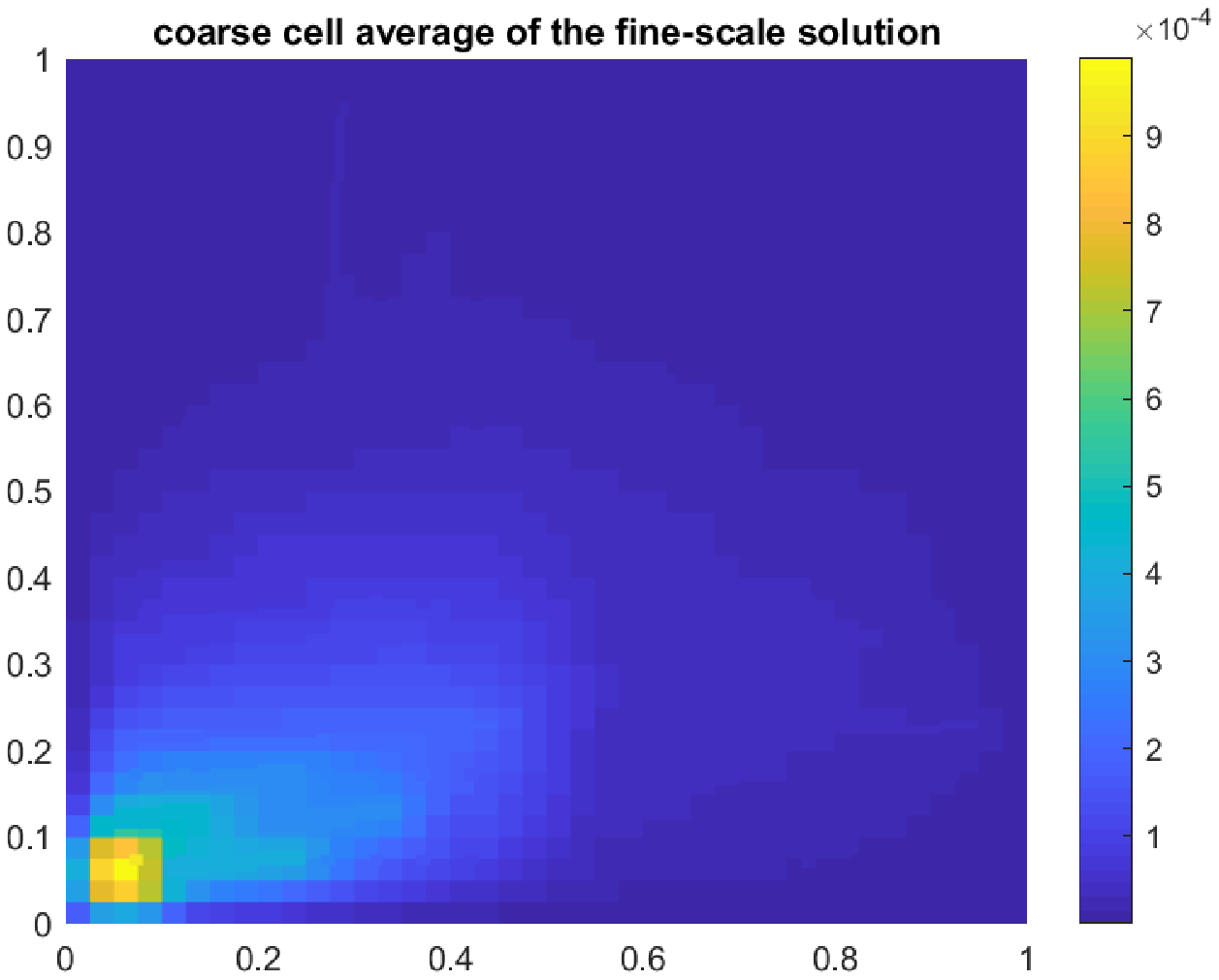}
}
\caption{coarse scale solution and coarse cell average of fine-scale solution.}
\label{ex3-solution-coarse40}
\end{figure}

Then in Table~\ref{table:con3} we display the relative $L^2$ error with respect to different coarse mesh sizes. With proper number of oversampling layers, the error converges as reported in Example~\ref{ex1}. Next, the relative $L^2$ error for coarse grids $20\times 20$ and $40\times 40$ with respect to different number of overampling layers are also reported in Table~\ref{ex3:con}, and this example once again highlights that the error decays as the oversampling layers increase, in addition, more oversampling layers are needed to obtain the desired error as the coarse mesh size decreases.
\begin{table}[H]
\centering
\begin{tabular}{l|llll}
\hline
$H$& oversampling coarse layers & $e_{L^2}$ \\
\hline
$\frac{1}{10}$& 3  & 0.0984\\
$\frac{1}{20}$& 4  & 0.0382\\
$\frac{1}{40}$& 5  & 0.0183\\
\hline
\end{tabular}
\caption{Relative $L^2$ error for Example~\ref{ex3} with varying coarse grid size.}
\label{table:con3}
\end{table}

\begin{table}[H]
\centering
\begin{tabular}{l|ll}
\hline
Layer& coarse mesh \;20$\times$ 20 &coarse mesh \;40$\times$40\\
\hline
1 & 0.8246     & 0.8429  \\
3 & 0.3070     & 0.7229  \\
4 & 0.0382     & 0.2408  \\
5 & 0.0025     & 0.0183   \\
6 & 1.2742e-4  & 5.337e-4 \\
\hline
\end{tabular}
\caption{Relative $L^2$ error with respect to different number of oversampling layers
for Example~\ref{ex3}.}
\label{ex3:con}
\end{table}

\section{Conclusion}\label{sec:conclusion}

In this paper we have developed a simple constraint energy minimization on the oversampling domain to generate the multiscale basis functions, where the construction of the multiscale basis functions relies on the scale separation. In addition, our theory illustrates that the number of oversampling layers required for the convergence is related to the local contrast ratio and the coarse mesh size $H$. Small contrast ratio in each region guarantees the convergence, thus, one should define proper regions in the numerical experiments in order to achieve the desired convergence. Two numerical examples are carried out to test the performances of the proposed method. The numerical results indicate that the relative error decays as the number of oversampling layers increases for a fixed coarse mesh size, furthermore, for a fixed oversampling size, the performance of the scheme will deteriorate as the medium contrast increases.

\section*{Acknowledgements}

The research of Eric Chung is partially supported by the Hong Kong RGC General Research Fund (Project numbers 14304217 and 14302018) and CUHK Faculty of Science Direct Grant 2017-18.


\end{document}